\documentclass[11pt, reqno]{amsart}
\usepackage{latexsym,amsmath,amssymb, amsthm, mathtools}
\usepackage{cases, verbatim, esint, todonotes}
\usepackage[active]{srcltx}
\usepackage{xcolor}
\usepackage{float}
\usepackage[normalem]{ulem}

\usepackage[colorlinks]{hyperref}

\hypersetup{
    bookmarks=true,         
    unicode=false,          
    pdftoolbar=true,        
    pdfmenubar=true,        
    pdffitwindow=false,     
    colorlinks=false,       
    linkcolor=red,          
    linkbordercolor=red,
    citecolor=green,        
    citebordercolor=green,
    filecolor=magenta,      
    urlcolor=cyan,           
   urlbordercolor={1 1 1}  
}

\usepackage[T1]{fontenc}
\usepackage{mathscinet}

\usepackage[margin=1.5in]{geometry}

\newtheorem{thm}{Theorem}
\newtheorem{lem}[thm]{Lemma}
\newtheorem{prop}[thm]{Proposition}
\newtheorem{cor}[thm]{Corollary}

\theoremstyle{remark}
\newtheorem{rmk}[thm]{Remark}
\newtheorem{example}[thm]{Example}

\theoremstyle{definition}
\newtheorem{defi}[thm]{Definition}

\numberwithin{thm}{section} 
\numberwithin{equation}{section}

\makeatletter

\newcommand{\Rmnum}[1]{\expandafter\@slowromancap\romannumeral #1@}
\makeatother

\def\A{{\mathcal A}}
\def\R{{\mathbb R}}
\def\X{{\mathbf X}}

\def\O{{\mathcal O}}

\newcommand{\pO}{\partial\Omega}
\newcommand{\Oba}{\overline{\Omega}}

\newcommand{\vep}{\varepsilon}
\newcommand{\ol}{\overline}
\newcommand{\ul}{\underline}

\newcommand{\bpm}{\begin{pmatrix}}
\newcommand{\epm}{\end{pmatrix}}

\newcommand{\beq}{\begin{equation}}
\newcommand{\eeq}{\end{equation}}
\def\lipl{{\rm Lip}_{loc}}

plus 6pt minus 12pt
\title[Eikonal equation in metric spaces]{\protect{Equivalence of solutions of eikonal equation\\ in metric spaces}}

\author[Q. Liu]{Qing Liu}
\address{Qing Liu, Department of Applied Mathematics, Fukuoka University, Fukuoka 814-0180, Japan, {\tt qingliu@fukuoka-u.ac.jp}}

\author[N. Shanmugalingam]{Nageswari Shanmugalingam}
\address{Nageswari Shanmugalingam, Department of Mathematical Sciences, University of Cincinnati, 
P.O.Box 210025, Cincinnati, OH 45221-0025, USA, {\tt shanmun@uc.edu}}

\author[X. Zhou]{Xiaodan Zhou}
\address{Xiaodan Zhou, Analysis on Metric Spaces Unit, Okinawa Institute of Science and Technology G
raduate University, 
Okinawa 904-0495, Japan\\ 
{\tt xiaodan.zhou@oist.jp}}

\date{\today}

\begin{document}

\begin{abstract}
In this paper we prove the equivalence  
between some known notions of 
solutions to the eikonal equation and more general 
analogs of the Hamilton-Jacobi equations in complete and rectifiably connected metric spaces. 
The notions considered are that of curve-based viscosity solutions,
slope-based viscosity solutions, and Monge solutions. By using the induced intrinsic (path)
metric, we reduce the metric space to a length space and show the equivalence of these solutions to the associated 
Dirichlet boundary 
problem. Without utilizing the boundary data, we also localize our argument and directly prove the equivalence for the 
definitions of solutions. Regularity of solutions related to the Euclidean semi-concavity is discussed as well.
\end{abstract}


\maketitle

\section{Introduction}

\subsection{Background and motivation}
In this paper, we are concerned with first order Hamilton-Jacobi equations in metric spaces. 
The Hamilton-Jacobi equations in the Euclidean spaces are widely applied in various fields such as 
optimal control,  geometric optics, computer vision, image processing.  It is well known that the notion 
of viscosity solutions provides a nice framework for the well-posedness of first order fully nonlinear 
equations; we refer to \cite{CIL, BC} for comprehensive introduction.  

In seeking to further develop various fields such as optimal transport \cite{AGS13, Vbook}, 
mean field games \cite{CaNotes}, topological networks \cite{SchC, IMZ,  ACCT,  IMo1, IMo2} etc., the 
Hamilton-Jacobi equations in a general metric space $(\X, d)$ have recently attracted great attention,
see for example~\cite{GaS, GHN}. 
Typical forms of the equations include
\beq\label{stationary eq}
H(x, u, |\nabla u|)=0 \quad \text{in $\Omega$,}
\end{equation}
and its time-dependent version 
\beq\label{evolution eq}
{\partial_t}u+H(x, t, u, |\nabla u|)=0 \quad \text{in $(0, \infty)\times \X$}
\end{equation}
with necessary boundary or initial value conditions. Here $\Omega\subsetneq \X$ is an open set and 
$H: \Omega\times \R\times \R\to \R$ is a given continuous function called the Hamiltonian of the Hamilton-Jacobi equation. 
While $\partial_t $ denotes the time differentiation, $|\nabla u|$ stands for a generalized notion of the gradient 
norm of $u$ in metric spaces.
 
In this paper  we also pay particular attention to the so-called eikonal equation
\beq\label{eikonal eq}
|\nabla u|(x)=f(x) \quad \text{in $\Omega$.}
\end{equation}
Here $f: \Omega\to [0, \infty)$ is a given continuous function satisfying
\[
\inf_{\Omega} f>0. 
\]
The eikonal equation in the Euclidean space has important applications in various fields such as geometric optics, 
electromagnetic theory and image processing \cite{KKbook, MSbook}.

Several new notions of viscosity solutions have recently been proposed in the general metric setting. 
We refer the reader to Section \ref{sec:review} for a review with precise definitions and basic properties of solutions,
and to the relevant papers we mention below.

Using rectifiable curves in the space $\X$, Giga, Hamamuki and Nakayasu~\cite {GHN} discussed a notion of metric 
viscosity solutions to \eqref{eikonal eq} and established well-posedness under the Dirichlet condition
\beq\label{bdry cond}
u=\zeta \quad \text{on $\pO$,}
\eeq 
where $\zeta$ is a given bounded continuous function on $\pO$ satisfying an appropriate regularity assumption to be 
discussed later. In the sequel, this type of solutions will be called curve-based solutions (or c-solutions for short). 
The definition of c-solutions (Definition \ref{defi c}) essentially relies on optimal control interpretations along 
rectifiable curves and requires very little structure of the space. The same approach is used in \cite{Na1} to 
study the evolution problem \eqref{evolution eq} with the Hamiltonian $H(x, )$ convex in $v$. Moreover, 
unique viscosity solutions of the eikonal equation in the sense of \cite{GHN} are also constructed on 
fractals like the Sierpinski gasket \cite{CCM}. 
 
On the other hand, when $(\X, d)$ is a complete geodesic space, by interpreting $|\nabla u|$ as the local slope of a locally Lipschitz function $u$, Ambrosio and Feng \cite{AF} provide a different viscosity 
approach to \eqref{evolution eq} for a class of convex Hamiltonians. This was extended to the class of
potentially nonconvex
Hamiltonians $H$ by
Gangbo and {\'S}wi{\polhk{e}}ch  \cite{GaS2, GaS}, who  proposed a generalized notion of viscosity solutions 
via appropriate test classes and proved uniqueness and existence of the solutions to more general 
Hamilton-Jacobi equations in length spaces. Stability and convexity of such solutions 
are studied respectively in \cite{NN} and in \cite{LNa}. Since this definition of solutions is based on 
the local slope, we shall call them slope-based solutions (or s-solutions for short) below. See the 
precise definition in Definition \ref{defi s}.

Since either approach above provides a generalized viscosity solution theory for first order 
nonlinear equations, especially the uniqueness and existence results, it is natural to expect that 
they actually agree with each other in the wider setting of metric spaces. This motivates 
us to explore the relations between both types of viscosity solutions and to understand  further 
connections to other possible approaches.

\subsection{Equivalence of solutions to the Dirichlet problem}
In the first part of this work (Section \ref{sec:cs}), we compare c- and s-solutions of the 
eikonal equation and show their equivalence when continuous solutions to the Dirichlet 
problem exist. To this end, we assume that $(\X, d)$ is complete and rectifably connected; 
namely, for any $x, y\in \X$, there exists a rectifiable curve in $\X$ joining $x$ and $y$. 

Our key idea is to use the induced intrinsic metric of the space $\X$, which is given by
\beq\label{int metric}
\tilde{d}(x, y)=\inf\{\ell(\xi): \xi\text{ is a rectifiable curve connecting }x\text{ and }y\}
\eeq
for $x, y\in \X$, where $\ell(\xi)$ denotes the length of the curve $\xi$ with respect to the original 
metric $d$. It is then easily seen that $(\X, \tilde{d})$ is a length space. Since this change of 
metric preserves the property of c-solutions, we can directly compare 
the notion of c-solutions in $(\X,d)$ with s-solutions in $(\X,\tilde{d})$. In order to preserve the 
completeness of the metric space,  we assume throughout the paper that
\beq\label{eq ast}
\tilde{d}\to 0 \ \text{as $d\to 0$}. 
\eeq

Before introducing our main results, we emphasize that in the sequel, by ``local'' we mean the 
property in question holds in sufficiently small open balls with respect to the intrinsic metric of the space.

Our first main result is as follows. 

\begin{thm}[Equivalence between c- and s-solutions]\label{thm equiv}
Let  $(\X, d)$ be a complete rectifiably connected metric space and $\tilde{d}$ be the intrinsic metric given by \eqref{int metric}. 
Assume that $(\X,d)$ also satisfies \eqref{eq ast}.
Suppose that $\Omega\subsetneq \X$ is an open set bounded with respect to the metric $\tilde{d}$. 
Assume that $f\in C(\Omega)$ satisfies $\inf_{\Omega} f>0$. 
If $u$ is a c-solution of \eqref{eikonal eq} with respect to $d$, 
then $u$ is a locally Lipschitz s-solution 
of \eqref{eikonal eq} with respect to $\tilde{d}$. 
In addition, if $f$ is uniformly continuous in $\Omega$, $\zeta$ is 
uniformly continuous on $\partial\Omega$, and
there exists a modulus of continuity $\sigma$ such that 
\beq\label{new bdry regularity}
|\zeta(y)-u(x)|\leq \sigma(\tilde{d}(x, y))\quad \text{for all $x\in \Omega$ and $y\in \pO$,}
\eeq
then $u$ is the unique 
c-solution and s-solution of \eqref{eikonal eq} satisfying~\eqref{new bdry regularity}. 
\end{thm}

In the second result of the theorem above,
we can replace the assumptions of uniform continuity of $u$ and $f$ by their continuity if  
$(\X, d)$ is additionally assumed to be proper, that is, any bounded closed 
subset of $\X$ is compact. 

As the metric $d$ was replaced with the metric $\tilde{d}$, the proof of the first statement in 
Theorem~\ref{thm equiv} is more or less analogous to the classical arguments to show the 
relation between a viscosity solution and its optimal control interpretation (cf. \cite{BC}). 

The c-solutions and $f$ considered in \cite{GHN} are in general only continuous along curves,  
 and therefore might not be continuous with respect to either the metric $d$ or the metric 
 $\tilde{d}$; see \cite[Example 4.9]{GHN}. 
 Here we assume the stronger requirement of continuity with respect to the metric $d$.
Then \eqref{new bdry regularity} in the second result of Theorem \ref{thm equiv} guarantees 
that the c-solution $u$ is uniformly continuous up to $\pO$. 
Recall that the comparison principle (and thus the 
uniqueness) for s-solutions \cite{GaS} needs such an assumption. More precisely, it was shown 
in \cite[Theorem 5.3]{GaS} that any s-subsolution $u$ and any s-supersolution $v$ 
satisfy $u\leq v$ in $\Oba$ whenever there exists a modulus $\sigma$ such that 
\beq\label{bdry verify0}
u(x)\leq \zeta(y)+\sigma(\tilde{d}(x, y)), \quad v(x)\geq \zeta(y)-\sigma(\tilde{d}(x, y))
\eeq
for all $x\in \Omega$ and $y\in \pO$. One thus needs to use \eqref{new bdry regularity} to 
validate the comparison principle and conclude the uniqueness and equivalence. 
 In 
Section~\ref{sec:bdry consistency}, for the sake of completeness, sufficient conditions for 
\eqref{new bdry regularity} are discussed.

\subsection{Local equivalence of solutions}

Our result in Theorem \ref{thm equiv} states that c- and s-solutions of \eqref{eikonal eq} coincide. 
However, we show the equivalence in presence of the Dirichlet boundary condition so as to use 
the comparison principle. One may wonder about a more direct proof of the equivalence  
without using the  boundary condition. 

The second part of this work is devoted to answering this question. Our method is based on  
localization of our arguments in the first part. To this end, we introduce a local version of the notion of c-solutions, 
which we call local c-solutions, by restricting the definition in a small metric ball centered at each 
point in $\Omega$; see Definition \ref{def local c}. We also include a third notion, called Monge 
solutions, in our discussion. We compare locally the notions of s-, local c-, and Monge s
olutions of \eqref{eikonal eq} in a complete length space. 

The Monge solution is known to be an alternative notion of solutions to Hamilton-Jacobi equations in 
Euclidean spaces \cite{NeSu, BrDa}. In a complete length space $(\X, d)$, our generalized 
definition of Monge solutions to \eqref{eikonal eq} is quite simple; it only requires a locally Lipschitz function $u$ to satisfy 
 \[
|\nabla^- u|(x)=f(x)\quad \text{for \emph{every} $x\in \Omega$,}
 \]
 where $|\nabla^- u|(x)$, given by 
 \[
 |\nabla^- u|(x)=\limsup_{y\to x} {\max\{u(x)-u(y), 0\}\over d(x, y)},
 \]
denotes the sub-slope of $u$ at $x$; see also the definition in \eqref{semi slope}. 
One advantage of this notion is that it does not involve any viscosity tests and the 
comparison principle can be easily established. We show that locally uniformly 
continuous c-, s- and Monge solutions of the eikonal equation are actually equivalent under a weaker 
positivity assumption on $f$. A more precise statement is given below. 

\begin{thm}[Local equivalence between solutions of eikonal equation]\label{thm equiv Monge}
Let $(\X,  d)$ be a complete length space and $\Omega\subset \X$ be an 
open set. 
Assume that $f$ is locally uniformly continuous and $f>0$ in $\Omega$.  
Let $u\in C(\Omega)$. Then the following statements are equivalent: 
\begin{enumerate}
\item[(a)] $u$ is a local c-solution of \eqref{eikonal eq};
\item[(b)] $u$ is a locally uniformly continuous s-solution of \eqref{eikonal eq};
\item[(c)] $u$ is a Monge solution of \eqref{eikonal eq}. 
\end{enumerate}
In addition, if any of (a)--(c) holds, then $u$ is locally Lipschitz with 
\beq\label{regular0}
|\nabla u|(x)=|\nabla^- u|(x)=f(x)\quad \text{for all  $x\in \Omega$.}
\eeq
\end{thm}

We actually prove more: 
\begin{itemize}
\item The notions of all locally uniformly continuous subsolutions are equivalent and 
locally Lipschitz (Proposition \ref{prop eikonal sub1}, Proposition \ref{prop eikonal sub2}); 
\item The notions of locally Lipschitz s-supersolutions and Monge supersolutions are 
equivalent (Proposition \ref{prop eikonal super}(i)); 
\item Any locally Lipschitz local c-supersolution is a Monge supersolution (Proposition \ref{prop eikonal super}(ii)); 
\item Any Monge solution is a local c-solution (Proposition \ref{prop eikonal solution}). 
\end{itemize}
We however do not know whether an s- or Monge supersolution needs to be a local c-supersolution.

In proving Theorem \ref{thm equiv Monge}, it turns out that the local Lipschitz continuity of solutions 
is an important ingredient.  Note that the local Lipschitz continuity holds for Monge solutions by 
definition, and it can be easily deduced for c-subsolutions as well because the space $(\X,\tilde{d})$ is
a length space, as shown in Lemma~\ref{lem c-lip}. 
In contrast, the Lipschitz regularity of s-solutions of \eqref{eikonal eq} is  less straightforward. 
Our proof requires the assumption on the local uniform continuity of s-solutions and $f$ due to 
possible lack of compactness for general length spaces. As stated in Corollary~\ref{cor equiv Monge}, 
we can remove such an assumption if $(\X, d)$ is proper (and therefore 
the length space space $X$ is a geodesic space because of
the generalized Hopf-Rinow theorem in metric spaces \cite{Gr, BHK}, which can be viewed as an immediate
consequence of the Arzela-Ascoli theorem). The notions of continuity and 
uniform continuity in a compact set are clearly equivalent. 

In this work we will assume that $(\X,d)$ satisfies the hypotheses of Theorem~\ref{thm equiv}. 
It is worth stressing that in this section $(\X, d)$ is assumed to be a length space only for 
simplicity from the point of view of c-solutions. The s-solutions and Monge solutions are not
defined for the case that $\X$ is not a length space. For metric spaces that are not length
spaces, the slopes in the definition of s-solutions and Monge solutions require us to
replace the metric $d$ with $\tilde{d}$.

It is not difficult to extend our discussion on the equivalence to the general 
equation~\eqref{stationary eq} under a monotonicity assumption on $p\to H(x, r, p)$. 
When $p\to H(x, r, p)$ is increasing,  we can simply apply an implicit-function-type 
argument to locally reduce the problem to the eikonal equation.

Besides, as in \eqref{regular0} in Theorem \ref{thm equiv Monge}, in this 
general case for any solution $u$ we can obtain the continuity of $|\nabla u|$ as well as the following property:
\beq\label{regular}
|\nabla u|(x)=|\nabla^- u|(x)\quad \text{for all  $x\in \Omega$;}
\eeq
in other words, the solution itself actually lie in the test class for s-subsolutions proposed 
in \cite{GaS2, GaS}. This type of properties also appears in the study of time-dependent 
Hamilton-Jacobi equations on metric spaces (cf. \cite{LoVi}). Our analysis reveals 
that \eqref{regular} resembles the semi-concavity in the Euclidean space. In fact, in the 
Euclidean space, \eqref{regular} implies the existence of a $C^1$ test function everywhere 
from above, which is a typical property of semi-concave functions. We expect more 
applications of the regularity property \eqref{regular}, since in general metric spaces 
defining convex functions is not trivial at all. It would be interesting to see further properties 
on such regular solutions in relation to the structure of PDEs and the geometric property.

The rest of the paper is organized as follows. In Section \ref{sec:review}, 
we review the definitions and properties of c- and s-solutions of Hamilton-Jacobi 
equations in metric spaces. Section \ref{sec:cs} is devoted to the proof of Theorem~\ref{thm equiv}. 
In Section~\ref{sec:monge} we propose the notion of Monge solutions and  prove 
Theorem~\ref{thm equiv Monge}.

\subsection*{Acknowledgements}

The work of the first author was supported by JSPS Grant-in-Aid for Scientific 
Research (No. 19K03574). 
The work of the second author is partially supported by
the grant DMS-\#1800161 of the National Science Foundation (U.S.A.)

\section{Metric viscosity solutions of Hamilton-Jacobi equations}\label{sec:review}

In this section we review the two notions of viscosity solutions to Hamilton-Jacobi equations in metric spaces, mentioned
in the first section. We focus on the stationary equation \eqref{stationary eq} and particularly the eikonal equation \eqref{eikonal eq}.

In the setting of general metric spaces, one needs to have an analog of $|\nabla u|$. To do so,
let us recall the definitions of  viscosity solutions in \cite{GHN} and \cite{GaS}.

\subsection{Curve-based solutions}

Given an interval $I=[a,b]\subset\R$ and a continuous function (curve) $\xi:I\to\Omega$, we define the length 
of $\gamma$ by
\[
\ell(\xi):=\sup_{a=t_0<t_1<\cdots<t_k=b} \sum_{j=0}^{k-1}d(\xi(t_j),\xi(t_{j+1})).
\]
Note that if $\ell(\xi)<\infty$, then the real-valued function $s_\xi:[a,b]\to[0,\ell(\gamma)]$ defined by
\[
s_\xi(t):=\ell(\xi\vert_{[a,t]})
\]
is a monotone increasing function, and hence it is differentiable at almost every $t\in[a,b]$. We denote
$s_\xi'$ by $|\xi^\prime|$, which can be equivalently defined by 
\[
|\xi^\prime|(t)=\lim_{\tau\to 0} {d(\xi(t+\tau), \xi(t))\over |\tau|}
\]
for almost every $t\in (a, b)$. Note that if $\xi$ is an absolutely continuous curve
(and so $s_\xi$ is absolutely continuous real-valued function), then
\[
s_\xi(t)=\int_a^t|\xi^\prime|(\tau)\, d\tau
\]
for all $t\in [a,b]$.

For any interval $I\subset \R$, we say an absolutely continuous curve $\xi: I\to \X$ is {admissible} if 
\[
|\xi'|\leq 1\quad  \text{a.e. in $I$.}
\]
Note that these are $1$-Lipschitz curves.
Let $\A(I, \X)$ denote the set of all admissible curves in $\X$
defined on $I$; without loss of generality,
we only consider intervals $I$ for which $0\in I$. For any $x\in \X$, 
we write $\xi\in \A_x(I, \X)$ if  
$\xi(0)=x$. We also need to define the exit time and entrance time of a curve $\xi$:
\begin{equation}\label{exit/entrance}
\begin{aligned}
T^+_\Omega[\xi]&:=\inf\{t\in I: t\geq 0, \xi(t)\notin \Omega\};\\
T^-_\Omega[\xi]&:=\sup\{t\in I: t\leq 0, \xi(t)\notin \Omega\}.
\end{aligned}
\end{equation}
Since any absolutely continuous curve is rectifiable and one can always reparametrize a  
rectifiable curve by its arc length (cf. \cite[Theorem 3.2]{Haj1}), hereafter we do not distinguish 
the difference between an absolutely continuous curve and a rectifiable (or Lipschitz) curve. 
\begin{defi}[Definition 2.1 in \cite{GHN}]\label{defi c}
An upper semicontinuous (USC) function $u$ in $\Omega$ is called a {curve-based viscosity subsolution} or 
{c-subsolution} of \eqref{eikonal eq} if for any $x\in\Omega$ and $\xi\in \A_x(\R, \Omega)$, we have 
\beq\label{eq c-sub}
\left|\phi'(0)\right|\leq f(x)
\eeq
whenever $\phi\in C^1(\R)$  
such that $t\mapsto u(\xi(t))-\phi(t)$, $t\in\xi^{-1}(\Omega)$, 
 attains a local maximum at $t=0$.

A lower semicontinuous (LSC) function $u$ in $\Omega$ is called a 
{curve-based viscosity supersolution} or  
{c-supersolution} of \eqref{eikonal eq} if 
for any $\vep>0$ and $x\in \Omega$, there exists $\xi\in \A_x(\R, \X)$ and $w\in LSC(T^-, T^+)$ with 
$-\infty<T^{\pm}=T^{\pm}_\Omega[\xi]<\infty$ such that 
\begin{equation}\label{wapprox}
w(0)=u(x), \quad w\geq u\circ\xi-\vep,
\end{equation}
and 
\begin{equation}\label{sgapprox}
\left|\phi'(t_0)\right|\geq f(\xi(t_0))-\vep
\end{equation}
whenever  $\phi\in C^1(\R)$ such that $w(t)-\phi(t)$ attains a minimum at $t=t_0\in (T^-, T^+)$. 
A function $u\in C(\Omega)$ is said to be a 
{curve-based viscosity solution} or 
{c-solution} if it is both a c-subsolution and a 
c-supersolution of \eqref{eikonal eq}. 
\end{defi}

In the definition of supersolutions, in general we cannot merely replace $w$ with
$u\circ\xi$. Suppose that $\X$ is not a geodesic space but a length space. When $f\equiv 1$, as 
we expect that the distance function $u=d(\cdot, x_0)$ is still a solution for any $x_0\in \X$, the 
supersolution property for $u\circ\xi$ without approximation would imply that $\xi$ is a geodesic, 
which is a contradiction.

The regularity of $u$ above can be  relaxed, since we only need its semicontinuity along each curve $\xi$.  
In fact, one can require a c-subsolution (resp., c-supersolution) to be merely arcwise upper  (resp., lower) 
semicontinuous; consult \cite{GHN} for details. However, in order to obtain our main results in 
this paper, we need to impose the conventional continuity of $u$ rather than the arcwise continuity.

The notions of c-subsolutions and c-supersolutions with respect to the metric $d$ and the intrinsic metric $\tilde{d}$ given in \eqref{int metric} are equivalent.  Note that $\xi$ is a curve with respect to $d$ if and only if it is a curve with
respect to $\tilde{d}$ because of our assumption~\eqref{eq ast}. It can also be seen that the speed $|\xi'|$ of the curve remains the same in both metrics; see Lemma~\ref{lem length}.  Hence, the class of admissible curves in Definition \ref{defi c} does not depend on the choice between $d$ and $\tilde{d}$. 



Although there seems to be no requirement on the metric space in the definition above,  
it is implicitly assumed in the definition of the c-supersolution that each point $x\in \Omega$ can be 
connected to the boundary $\partial\Omega$ by a curve of finite length. 

Uniqueness of c-solutions of \eqref{eikonal eq} with boundary data \eqref{bdry cond} is shown 
by proving a comparison principle \cite[Theorem 3.1]{GHN}. The existence of solutions in $C(\Oba)$, 
on the other hand, is based on an optimal control interpretation; in particular, it is 
shown in~\cite[Theorem~4.2 and Theorem~4.5]{GHN} that 
\beq\label{eq optimal control}
u(x)=\inf_{\xi\in C_x} \left\{ \int_{0}^{T_\Omega^+[\xi]} f(\xi(s))\, ds+\zeta\left(\xi(T_\Omega^+[\xi])\right)\right\}
\eeq
is a c-solution of \eqref{eikonal eq} and \eqref{bdry cond} provided that for each $x\in\Oba$ we have 
\[
C_x:=\left\{\xi\in \mathcal{A}_x([0, \infty), \X) \, :\, T_\Omega^+[\xi]\in (0, \infty)\right\}\neq \emptyset
\]
and $\zeta$ satisfies a boundary regularity, see \eqref{bdry regularity2} 
below. 

Note in the definition of c-supersolution, for each $(x,\vep)$ the conditions \eqref{wapprox} 
and \eqref{sgapprox} for $(\xi, w)$ are satisfied for all $t\in (T^-, T^+)$. We localize this definition as follows. 
Recall the notions of $T^{\pm}_\O[\xi]$ for open sets $\O\subset\X$ from~\eqref{exit/entrance}.

\begin{defi}[Local curve-based solutions]\label{def local c}
A function $u\in LSC(\Omega)$ is said to be a local
{c-supersolution} if for each $x\in\Omega$ there exists $r>0$ with $B_r^d(x)\subset\Omega$, and
for each $\vep>0$ we can find a curve $\xi_\vep\in \A_x(\R, \X)$ with $\xi_\vep(0)=x$ and a function 
$w\in LSC(t_r^-, t_r^+)$ with $t^-_r:=T^-_{B_r^d(x)}[\xi_\vep]$ and $t^+_r:=T^+_{B_r^d(x)}[\xi_\vep]$ such that
\[
w(0)=u(x), \quad w(t)\geq u\circ\xi_\vep(t)-\vep \ \text{ for all } t\in (t_r^-,t_r^+),
\]
and 
\[
\left|\phi'(t_0)\right|\geq f(\xi_\vep(t_0))-\vep 
\]
whenever  $\phi\in C^1(\R)$ such that $w(t)-\phi(t)$ attains a minimum at $t=t_0\in (t_r^-, t_r^+)$. 
A function $u\in C(\Omega)$ is said to be a local
{c-solution} if it is both a c-subsolution and a 
local c-supersolution of \eqref{eikonal eq}. 
\end{defi}

In the definition of local c-supersolutions given above,
the ball $B_r^d(x)$ is taken with respect to $d$. If $(\X, d)$ is a complete rectifiably connected 
metric space such that the intrinsic metric $\tilde{d}$ defined in \eqref{int metric} satisfies the consistency 
condition \eqref{eq ast}, 
then it is equivalent to use metric balls $B_r(x)$
with respect to $\tilde{d}$. This definition is studied mainly in Section~4,
where we assume $(\X,d)$ to be a length space. For length spaces balls with respect to $d$ and balls with respect to $\tilde{d}$
are the same, that is, $B_r^d(x)=B_r(x)$.

A c-supersolution (resp., c-solution) is clearly a local c-supersolution (resp., local c-solution), but it is not clear to us whether 
the reverse is also true in general. The notion of c-subsolutions is already a localized one, and we therefore do not have to define
``local c-subsolutions'' separately.

\subsection{Slope-based solutions}
We next discuss the definition proposed in \cite{GaS2}, which relies more on the property of 
geodesic or length metric. We denote by $\text{Lip}_{loc}(\Omega)$  
the set of locally Lipschitz 
continuous functions on an open subset $\Omega$ of a complete length space $(\X, d)$. 
For $u\in \text{Lip}_{loc}(\Omega)$ and for $x\in\Omega$, we define the local slope of $u$ to be
\beq\label{slope}
|\nabla u|(x):=\limsup_{y\to x}\frac{|u(y)-u(x)|}{d(x,y)}.
\eeq
Let 
\beq\label{test class}
\begin{aligned}
\overline{\mathcal{C}}(\Omega) &:= \{ u \in \text{Lip}_{loc}(\Omega) \, :\,  \text{$|\nabla^+ u| 
= |\nabla u|$ and $|\nabla u|$ is continuous in $\Omega$} \}, \\
\underline{\mathcal{C}}(\Omega) &:= \{ u \in \text{Lip}_{loc}(\Omega) \, :\,  \text{$|\nabla^- u| 
= |\nabla u|$ and $|\nabla u|$ is continuous in $\Omega$} \}, \\
\end{aligned}
\eeq
where, for each $x\in \X$, 
\beq\label{semi slope}
|\nabla^\pm u|(x) := \limsup_{y \to x} \frac{[u(y)-u(x)]_\pm}{d(x, y)}
\eeq
with $[a]_+:=\max\{a, 0\}$ and $[a]_-:=-\min\{a, 0\}$ for any $a\in \R$. In this work we call 
$|\nabla^+  u|$ and $|\nabla^- u|$ the (local) super- and sub-slopes of $u$ respectively; they 
are also named super- and sub-gradient norms in the literature (cf. \cite{LoVi}). 

Concerning the test classes $\ol{\mathcal{C}}(\Omega)$ and $\ul{\mathcal{C}}(\Omega)$, it is known from
\cite[Lemma 7.2]{GaS2} and \cite[Lemma 2.3]{GaS} that in a  length space $X$,  
$Ad(\cdot , x_0)^2$ belongs to $\ul{\mathcal{C}}(\Omega)$ for for any 
$x_0\in \X$ and $A>0$; moreover, $Ad(\cdot , x_0)^2$ belongs to 
$\ol{\mathcal{C}}(\Omega)$ for any 
$x_0\in \X$ and $A<0$.

Now we recall from \cite{GaS}  the definition of s-solutions of a general class of Hamilton-Jacobi equations. 

\begin{defi}[Definition 2.6 in \cite{GaS}]\label{defi s}
An USC (resp., LSC) function $u$ in an open set $\Omega\subset \X$ is called a
{slope-based viscosity subsolution} (resp., {slope-based viscosity supersolution}) or 
 {s-subsolution} (resp., {s-supersolution}) of \eqref{stationary eq} if
\beq\label{s-sub eq}
H_{|\nabla \psi_2|^*(x)}(x, u(x),  |\nabla \psi_1|(x)) \le 0
\end{equation}
\beq\label{s-sup eq}
\left(\text{resp., }\quad H^{|\nabla \psi_2|^*(x)}(x, u(x), |\nabla \psi_1|(x)) \ge 0\right)
\end{equation}
holds for any $\psi_1 \in \underline{\mathcal{C}}(\Omega)$ (resp., $\psi_1 \in \overline{\mathcal{C}}(\Omega)$) 
and $\psi_2 \in \text{Lip}_{loc}(\Omega)$  such that $u-\psi_1-\psi_2$ attains a local maximum  (resp., minimum) 
at a point $x \in \Omega$, where, for any $(x, \rho, p)\in \Omega\times \R\times \R$ and $a>0$,
\[
H_a(x, \rho, p) = \inf_{|q-p| \le a}H(x, \rho,  q), \quad H^a(x, \rho, p) = \sup_{|q-p| \le a}H(x, \rho, q) \quad \text{ for $a\geq 0$,}
\] 
and $|\nabla \psi_2|^*(x) = \limsup_{y \to x}|\nabla \psi_2|(y)$. We say that $u\in C(\Omega)$ is an 
s-solution of \eqref{stationary eq} if it is both an s-subsolution and an s-supersolution of \eqref{stationary eq}. 
\end{defi}

In the case of \eqref{eikonal eq}, we can define subsolutions (resp.,  supersolutions) by 
replacing \eqref{s-sub eq} (resp., \eqref{s-sup eq}) with 
\beq\label{s-sub eikonal}
|\nabla\psi_1|(x)\leq f(x)+|\nabla \psi_2|^\ast(x)
\eeq
\beq\label{s-super eikonal}
\left(\text{resp.,} \quad |\nabla\psi_1|(x)\geq f(x)-|\nabla \psi_2|^\ast(x) \right).
\eeq

When $\X=\R^N$, it is not difficult to see that $C^1(\Omega)\subset \ol{\mathcal{C}}(\Omega)\cap \ul{\mathcal{C}}(\Omega)$ 
for any open set $\Omega\subset \X$. Hence,  s-subsolutions, s-supersolutions and s-solutions of \eqref{stationary eq} 
in this case reduce to  conventional viscosity subsolutions, supersolution and solutions respectively. 

Concerning the test functions in a general geodesic or length space $(\X, d)$, it is known that, for any $k\geq 0$, 
$x_0\in \X$, the function $x\mapsto k\varphi(d(x, x_0))$ (resp., $x\mapsto -k\varphi(d(x, x_0))$) belongs to the 
class $\ul{\mathcal{C}}(\Omega)$ (resp., $\ol{\mathcal{C}}(\Omega)$) provided that $\varphi\in C^1([0, \infty))$ 
satisfies $\varphi'(0)=0$ and $\varphi'\geq 0$; see details in \cite[Lemma 7.2]{GaS2} and \cite[Lemma 2.3]{GaS}. 

Comparison principles for s-solutions are given in \cite[Theorem 5.3]{GaS} for the eikonal equation and 
in \cite[Theorem 5.1, Theorem 5.3]{GaS} for more general Hamilton-Jacobi equations.

\section{Equivalence between curve- and metric-based solutions}\label{sec:cs}

We give a proof of Theorem \ref{thm equiv} in this section. Let us begin with some elementary results on the 
space $\X$ with the intrinsic metric $\tilde{d}$. Then we show that any c-subsolutions and c-supersolutions 
are respectively s-subsolutions and s-supersolutions. 
\subsection{Metric change}
Let $(\X, d)$ be a complete rectifiably connected metric space. We compare the notions of viscosity 
solutions to \eqref{eikonal eq} provided respectively in 
Definition \ref{defi c} and in Definition \ref{defi s}. The key to our argument is 
to use \eqref{int metric} to connect the geometric setting of two notions.

It is clear that
\beq\label{metric comparison}
d(x, y)\leq \tilde{d}(x, y)\quad \text{for all }x, y\in \X.
\eeq
Therefore bounded sets in $(\X, \tilde{d})$ are bounded in $(\X, d)$. 
Under the assumption of rectifiable connectedness of $(X, d)$, we also see that $\tilde{d}(x, y)<\infty$ 
for any $x, y\in \X$ and $\tilde{d}$ is a metric on $\X$. Moreover, by \eqref{metric comparison}, it is 
clear that open sets in $(\X, d)$ are also open in $(\X, \tilde{d})$.

The metric $\tilde{d}$ is also used in \cite{GHN} to study the 
continuity and stability of c-solutions.
In the rest of this work,  $B_r(x)$ denotes the open ball centered at $x\in \X$ with radius $r>0$ with 
respect to the intrinsic metric $\tilde{d}$. 

The induced intrinsic structure leads us to the following elementary fact that $(\X, \tilde{d})$ is a length space. 
One can find this classical result in \cite[Proposition 2.3.12]{BBIbook} and \cite[Corollary 2.1.12]{Pa} for instance. 
See also~\cite{DJS} for the arc-length parametrization of rectifiable curves with respect to $d$ and
with respect to $\tilde{d}$.

\begin{lem}[Length space under intrinsic metric]\label{lem length}
Assume that $(\X, d)$ is a complete rectifiably connected metric space.
Let $\tilde{d}$ be the intrinsic metric of a metric space $(\X, d)$ as defined in \eqref{int metric}. 
Then $(\X, \tilde{d})$ is a length space. Moreover, for any rectifiable curve $\xi$, $\xi(s): I\to \X$ is a 
parametrization with respect to $d$ if and only if it is a parametrization with respect to $\tilde{d}$ and the speed $|\xi'|$ in both metrics coincide.
\end{lem}

%
We remark that a similar intrinsic metric is constructed in \cite{DeP1, DeP2, DJS} involving a given 
measure on the space. Ours is standard and simpler, since measures do not play a role in the current work.

The completeness of the metric space is needed to properly define s-solutions under the metric 
$\tilde{d}$. Note that $(\X, \tilde{d})$ is complete if $(\X, d)$ is complete, since $d\leq \tilde{d}$ holds
and $(\X,d)$ satisfies the condition~\eqref{eq ast}.

\subsection{Equivalence between solutions of the Dirichlet problem}

Let us start proving Theorem \ref{thm equiv}. 
We first prove that any c-subsolution is an s-subsolution. We need the following characterization of 
c-subsolutions given in \cite{GHN}. 

\begin{prop}[Proposition 2.6 in \cite{GHN}]\label{prop c-sub}
Assume that $f\in C(\Omega)$ with $f\ge 0$ in $\Omega$. 
Let $u$ be upper semicontinuous in $\Omega$. Then the following statements are equivalent:
\begin{itemize}
\item[(1)] $u$ is a c-subsolution of \eqref{eikonal eq} in $(\Omega, d)$.\\ 
\item[(2)] The inequality 
\beq\label{prop c-sub eq}
u(\xi(t_1))\leq \int_{t_1}^{t_2} f(\xi(r))\, dr+u(\xi(t_2))
\eeq
for all $\xi\in \mathcal{A}(\mathbb{R}, \Omega)$ and $t_1, t_2\in \mathbb{R}$ with $t_1<t_2$.
\end{itemize}
\end{prop}

Such characterization enables us to deduce the local Lipschitz continuity of c-subsolutions with respect to the metric $\tilde{d}$.

\begin{lem}[Local Lipschitz continuity of c-subsolutions]\label{lem c-lip}
Let  $(\X, d)$ be a complete rectifiably connected metric space and $\tilde{d}$ be the intrinsic metric 
given by \eqref{int metric}. Assume that \eqref{eq ast} holds. 
Let $\Omega\subsetneq \X$ be an open set.  
Assume that $f\in C(\Omega)$ and $f\geq 0$  in $\Omega$. 
If $u$ is upper semicontinuous in $\Omega$ and is a c-subsolution of \eqref{eikonal eq} in $(\Omega, d)$, then 
$u\in \lipl(\Omega)$. In particular, for any $x_0\in \Omega$ and $r>0$ such that 
$B_{2r}(x_0)\subset \Omega$ and $f$ is bounded in $B_{2r}(x_0)$, $u$ satisfies
\beq\label{local lip precise}
|u(x)-u(y)|\leq \tilde{d}(x, y)\sup_{B_{2r}(x_0)} f \quad \text{for all $x, y\in B_r(x_0)$.}
\eeq
\end{lem}       

\begin{proof}
Fix arbitrarily $x_0\in \Omega$ and $r>0$ small such that $B_{2r}(x_0)\subset \Omega$ (with 
respect to $\tilde{d}$) and $f$ is bounded on $B_{2r}(x_0)$. 
For any $x, y\in B_r(x_0)$ and any $0<\vep< 2r-\tilde{d}(x, y)$, there exists 
an arc-length parametrized rectifiable curve $\xi_0$ joining $x$ and $y$ satisfying 
\[
\ell(\xi_0)\leq \tilde{d}(x, y)+\vep< 2r. 
\]
It follows that $\xi_0\subset B_{2r}\subset \Omega$. Applying the characterization of c-subsolutions in 
Proposition~\ref{prop c-sub}(ii) with $\xi=\xi_0$, we have 
\[
u(x)-u(y)\leq (\tilde{d}(x, y)+\vep)\sup_{B_{2r}(x_0)} f. 
\]
Passing to the limit $\vep\to 0$, we get 
\[
u(x)-u(y)\leq \tilde{d}(x, y)\sup_{B_{2r}(x_0)} f.
\]
Exchanging the roles of $x$ and $y$, we thus obtain \eqref{local lip precise}. 
\end{proof}

We next continue to use Proposition \ref{prop c-sub} to show that c-subsolutions of \eqref{eikonal eq} are 
s-subsolutions with respect to the metric $\tilde{d}$.
 
\begin{prop}[Implication of subsolution property]\label{prop sub}
Let  $(\X, d)$ be a complete rectifiably connected metric space and $\tilde{d}$ be the intrinsic metric 
given by \eqref{int metric}.  Assume that \eqref{eq ast} holds. 
Let $\Omega\subsetneq\X$ be an open set. Assume that $f\in C(\Omega)$ and  $f\geq 0$  in $\Omega$. 
If $u$ is upper semicontinuous in $\Omega$ and is a c-subsolution of \eqref{eikonal eq} in $(\Omega, d)$, then $u$ is an 
s-subsolution of \eqref{eikonal eq} in $(\Omega, \tilde{d})$. 
\end{prop}

\begin{proof}
Since $(\X, \tilde{d})$ is a length space, our notation $\text{Lip}_{loc}(\Omega)$ now denotes the set of all 
locally Lipschitz functions on $\Omega$ with respect to the
intrinsic metric $\tilde{d}$. Note that if $u$ is upper semicontinuous with respect to the metric $d$, 
then it is upper semicontinous with respect to $\tilde{d}$, since $\tilde{d}\to 0$ if and only if $d\to 0$ 
due to \eqref{int metric} and \eqref{eq ast}.

Fix $x_0\in \Omega$ arbitrarily.  Assume that there exists $\psi_1 \in \underline{\mathcal{C}}(\Omega)$ and 
$\psi_2 \in \text{Lip}_{loc}(\Omega)$ such that $u-\psi_1-\psi_2$ attains a local maximum at a point $x_0$. 
So there is some $r_0>0$ with $B_{2r_0}(x_0)\subset\Omega$ such that 
\[
u(x)-u(x_0)\leq (\psi_1+\psi_2)(x)-(\psi_1+\psi_2)(x_0)
\]
for all $x\in B_{r_0}(x_0)$. 

Moreover, for any fixed $\vep\in (0, 1)$, by the continuity of $f$, we can make $0<r_0$ smaller so that  
\begin{equation}\label{general curve2}
|f(x)-f(x_0)|\leq \vep
\end{equation}
if $x\in B_{r_0}(x_0)$. We fix such $r_0>0$ (and keep in mind that $r_0$ now also depends on $\vep$). 

For any $r\in (0, r_0/2)$ and any $x\in \Omega$ with $0<\tilde{d}(x, x_0)<r$, there exists an arc-length 
parametrized curve $\xi$ in $\Omega$ such that $\xi(0)=x_0$ and $\xi(t)=x$, where 
\begin{equation}\label{general curve0}
t=\ell(\xi) \leq \tilde{d}(x, x_0)+\vep\tilde{d}(x, x_0). 
\end{equation}
Applying Proposition \ref{prop c-sub} for such a curve with $t_1=0, t_2=t$, we get 
\[
u(x_0)\leq \int_0^{t} f(\xi(s))\, ds+u(x),
\]
and therefore,
\[
(\psi_1+\psi_2)(x_0)-(\psi_1+\psi_2)(x)\leq \int_0^{{t}}f(\xi(s))\, ds.
\]
Dividing the inequality above by $\tilde{d}(x,x_0)$, we get
\begin{equation}\label{general curve}
{\psi_1(x_0)-\psi_1(x)\over \tilde{d}(x, x_0)}
\leq \frac{1}{\tilde{d}(x, x_0)}\int_0^{{t}}f(\xi(s))\, ds+{\psi_2(x)-\psi_2(x_0)\over \tilde{d}(x, x_0)}.
\end{equation}
Since $\vep<1$ and $r<r_0/2$, we have
\[
\tilde{d}(\xi(s), x_0)\leq t\leq r+\vep r<r_0
\]
for all $s\in [0, t]$, by \eqref{general curve2}. Therefore 
\[
f(\xi(s))\leq f(x_0)+\vep
\]
for all $s\in [0, t]$. Hence \eqref{general curve} yields 
\[
{\psi_1(x_0)-\psi_1(x)\over \tilde{d}(x, x_0)}\leq \frac{t}{\tilde{d}(x, x_0)}(f(x_0)+\vep)+{\psi_2(x)-\psi_2(x_0)\over \tilde{d}(x, x_0)}.
\]
Using \eqref{general curve0} and recalling that the choice of $r_0$ depends on $\vep$, we thus have
\begin{equation}\label{general curve1}
{\psi_1(x_0)-\psi_1(x)\over \tilde{d}(x, x_0)}\leq (1+\vep)(f(x_0)+\vep)+{\psi_2(x)-\psi_2(x_0)\over \tilde{d}(x, x_0)}
\end{equation}
for all $\vep>0$ and all $x\in \Omega$ with $x\in B_{r_0}(x_0)$. 

Since $\psi_1\in \underline{\mathcal{C}}(\Omega)$, 
there exists a sequence of points $x_n\in \Omega$ such that, as $n\to \infty$, we have $x_n\to x_0$ and 
\[
{\psi_1(x_0)-\psi_1(x_n)\over \tilde{d}(x_n, x_0)}\to |\nabla^-\psi_1|(x_0)=|\nabla \psi_1|(x_0).
\]
Adopting \eqref{general curve1} with $x=x_n$ and sending $n\to \infty$ and then $\vep\to 0$, we end up with
the desired inequality \eqref{s-sub eikonal} at $x=x_0$.
\end{proof}

We next show that any c-supersolution is an s-supersolution. 
We again use a result presented in \cite{GHN}. 

\begin{prop}[Proposition 2.8 in \cite{GHN}]\label{prop c-super}
Assume that $f\in C(\Omega)$ with $f\ge 0$. 
Assume $\inf_{\Omega}f>0$. Let $u$ be a lower semicontinuous
c-supersolution of \eqref{eikonal eq}. Then for any $\vep>0$ and $x_0\in \Omega$, there exists 
$\xi_\vep\in \mathcal{A}_{x_0}([0, \infty), \Omega)$ satisfying $T=T_\Omega^+[\xi_\vep]<\infty$ and 
\beq\label{prop c-super eq}
u(x_0)\geq \int_0^{t} f(\xi_\vep(s))\, ds+u(\xi_\vep(t))-\vep(1+t)
\eeq
for all $0\leq t\leq T$.
\end{prop}

{\begin{rmk}\label{local c-super prop}  
If $u$ is a local c-supersolution instead, then for each $x_0\in \Omega$ there is a sufficiently small 
$r>0$ such that for each $\vep>0$ we can find a choice $\xi_\vep\in\A_{x_0}([0,\infty), B_r(x_0))$
such that for all $0\le t\le T^+_{B_r(x_0)}[\xi_\vep]$, \eqref{prop c-super eq} holds. 
This is seen by directly adapting the proof of \cite[Proposition~2.8]{GHN}.
\end{rmk}

\begin{prop}[Implication of supersolution property]\label{prop super}
Let  $(\X, d)$ be a complete rectifiably connected metric space and $\tilde{d}$ be the intrinsic 
metric given by \eqref{int metric}. Assume that \eqref{eq ast} holds. 
Let $\Omega\subsetneq \X$ be an open set. 
Assume that $f\in C(\Omega)$ with $\inf_{\Omega}f>0$. 
If $u$ is a lower semicontinuous c-supersolution of \eqref{eikonal eq} in $(\Omega, d)$, 
then $u$ is an s-supersolution of \eqref{eikonal eq} in $(\Omega, \tilde{d})$. 
\end{prop}

\begin{proof}
Since $u$ is lower semicontinuous with respect to the metric $d$, it is easily seen that $u$ is also lower 
semicontinuous with respect to $\tilde{d}$, thanks to \eqref{int metric} and \eqref{eq ast}. 

Fix $x_0\in \Omega$ arbitrarily.  Assume that there exist $\psi_1 \in \overline{\mathcal{C}}(\Omega)$ and 
$\psi_2 \in \text{Lip}_{loc}(\Omega)$ such that $u-\psi_1-\psi_2$ attains a local minimum at a point $x_0$. We thus have 
$r_0>0$ such that $B_{2r_0}(x_0)\subset\Omega$ and
\[
u(y)-u(x_0)\geq (\psi_1+\psi_2)(y)-(\psi_1+\psi_2)(x_0)
\]
for all $y\in B_{r_0}(x_0)$.  

Applying Proposition \ref{prop c-super}, for any $\vep>0$ satisfying $\sqrt{\vep}<\min\{r,\tilde{d}(x_0, \pO)\}$, 
we can find $\xi\in \mathcal{A}_{x_0}([0, \infty), \Omega)$ such that \eqref{prop c-super eq} 
holds for all $0\leq t\leq T_\Omega^+[\xi]$. Since $T_\Omega^+[\xi]\geq \tilde{d}(x_0, \pO)>0$, 
we can take $t=\sqrt{\vep}$ and $x_\vep=\xi(\sqrt{\vep})$ in \eqref{prop c-super eq} to get
\[
(\psi_1+\psi_2)(x_\vep)-(\psi_1+\psi_2)(x_0)\leq -\int_0^{\sqrt{\vep}}f(\xi(s))\, ds+\vep(1+\sqrt{\vep}).
\]
Dividing this relation by $\sqrt{\vep}$, we get
\beq\label{eq lem super1}
{1\over \sqrt{\vep}}\int_0^{\sqrt{\vep}} f(\xi(s))\, ds+{\psi_2(x_\vep)-\psi_2(x_0)\over \sqrt{\vep}}-\sqrt{\vep}(1+\sqrt{\vep})\le 
{\psi_1(x_0)-\psi_1(x_\vep)\over \sqrt{\vep}}.
\eeq
Noticing that 
\[
\tilde{d}\left(x_\vep, x_0\right)\leq \sqrt{\vep},
\]
we deduce that 
\[
{\psi_1(x_\vep)-\psi_1(x_0)\over \sqrt{\vep}}\le {|\psi_1(x_\vep)-\psi_1(x_0)|\over \sqrt{\vep}}
\le {|\psi_1(x_\vep)-\psi_1(x_0)|\over \tilde{d}\left(x_\vep, x_0\right)}
\]
and
\[
{\psi_2(x_\vep)-\psi_2(x_0)\over \sqrt{\vep}}\ge -{|\psi_2(x_\vep)-\psi_2(x_0)|\over \sqrt{\vep}}
\ge -{|\psi_2(x_\vep)-\psi_2(x_0)|\over \tilde{d}\left(x_\vep, x_0\right)}.
\]
Hence, combining the above inequalities together implies that
\[
{|\psi_1(x_\vep)-\psi_1(x_0)|\over \tilde{d}\left(x_\vep, x_0\right)}
\geq {1\over \sqrt{\vep}}\int_0^{\sqrt{\vep}} f(\xi(s))\, ds- {|\psi_2(x_\vep)-\psi_2(x_0)|\over \tilde{d}\left(x_\vep, x_0\right)}-
\sqrt{\vep}(1+\sqrt{\vep}).
\]
Letting $\vep\to 0$,  we are led to \eqref{s-super eikonal} with $x=x_0$ as desired.
\end{proof}

\begin{rmk}\label{local c to s} 
By Remark \ref{local c-super prop}, it is not difficult to see that if $u$ is only a local c-supersolution, then for any 
$x_0\in \Omega$ the same result as in Proposition \ref{prop super} holds in $B_r(x_0)$ with $r>0$ small. In fact, 
the proof will be the same except that $\Omega$ should be replaced by $B_r(x_0)$. 
\end{rmk}

We now prove Theorem \ref{thm equiv}. 
\begin{proof}[Proof of Theorem \ref{thm equiv}]

In view of Lemma~\ref{lem c-lip}, we know that any c-solution of \eqref{eikonal eq} is locally Lipschitz with respect to $\tilde{d}$.
Now by Proposition~\ref{prop sub} and Proposition~\ref{prop super}
we see that any c-solution $u$ of \eqref{eikonal eq} is a locally Lipschitz s-solution. 
If $u$ satisfies \eqref{new bdry regularity} and $f$ is uniformly continuous with $\inf_{\Omega}f>0$, 
then in the bounded domain $\Omega$ we can apply the comparison principle for 
s-solutions (cf. \cite[Theorem 5.3]{GaS}) to show that $u$ must be the only s- and c-solution of the Dirichlet problem. 
\end{proof}

\subsection{Boundary value}\label{sec:bdry consistency}

In light of the second part of Theorem~\ref{thm equiv}, the importance of the condition \eqref{new bdry regularity}
is clear. We now give sufficient conditions for \eqref{new bdry regularity}, which is also important for the 
existence of c-solutions. Recall that $\zeta$ is a continuous function on $\partial\Omega$, playing the role
of the Dirichlet boundary data in \eqref{bdry cond}.

\begin{prop}[Boundary consistency]\label{prop bdry regularity}
Assume that $(\X, d)$ is a complete rectifiably connected metric space and $\tilde{d}$ be the 
induced intrinsic metric given by \eqref{int metric}. Assume that \eqref{eq ast} holds. Let $\Omega\subsetneq \X$ be an open set. 
Suppose that $f\in C(\Oba)$ is bounded and $f\geq 0$ in $\Oba$. Let $u$ be given by \eqref{eq optimal control} 
with $\zeta\in C(\pO)$ given. 
\begin{enumerate}
\item[(1)] 
If there exists $L>0$ such that  $\zeta$ is $L$-Lipschitz on $\partial\Omega$
with respect to the metric $\tilde{d}$, then 
\beq\label{sol bdry regularity weak}
u(x)- \zeta(y)\leq \tilde{d}(x, y)\max\left\{L,\ \sup_{\Oba} f\right\}\quad \text{for all $x\in \Omega$ and $y\in \pO$.}
\eeq
\item[(2)] If $\zeta$ satisfies a stronger condition:
\beq\label{bdry regularity}
|\zeta(x)-\zeta(y)|\leq  \tilde{d}(x, y)  \inf_{\Oba} f \quad \text{for every $x, y\in \pO$,}
\eeq
then 
\beq\label{sol bdry regularity}
|u(x)- \zeta(y)|\leq \tilde{d}(x, y)\sup_{\Oba} f \quad \text{for all $x\in \Omega$ and $y\in \pO$.}
\eeq
\end{enumerate}
\end{prop}

\begin{proof}
For simplicity of notation, denote
\[
m:=\inf_{\Oba} f, \quad M:=\sup_{\Oba} f. 
\]
Fix $x\in \Omega$ and $y\in \pO$. Then for any $\vep>0$, there exists an arc-length parametrized curve 
$\xi\in \A_{x}(\R, \X)$ such that $\xi(t)=y$ and
\[
\tilde{d}(x, y)\leq t\leq \tilde{d}(x, y)+\vep.
\]
This curve may not stay in $\Omega$, but there exists $z=\xi(t_1)\in \pO$, where  
\[
t_1:=T^+_\Omega[\xi]=\inf\{s: \xi(s)\in \pO\}. 
\]
Since we have
\[
t-t_1\geq \tilde{d}(y, z) \ \text{ and }\  t\leq \tilde{d}(x, z)+\tilde{d}(y, z)+\vep, 
\]
it follows that $t_1\leq \tilde{d}(x, z)+\vep$ and therefore
\beq\label{bdry regu1}
\tilde{d}(x, z)+\tilde{d}(z,y)\leq t \leq \tilde{d}(x, y)+\vep. 
\eeq
Now in view of \eqref{eq optimal control}, we have
\beq\label{bdry regu rev1}
u(x)\leq \zeta(z)+\int_0^{t_1} f(\xi(s))\, ds\leq \zeta(z)+M(\tilde{d}(x, z)+\vep).
\eeq
Thanks to the $L$-Lipschitz continuity of $\zeta$, we have  
\beq\label{use later1}
\zeta(z)\leq \zeta(y)+L\tilde{d}(y, z). 
\eeq

Applying \eqref{bdry regu1} and \eqref{use later1} in \eqref{bdry regu rev1}, we thus get
\beq\label{use later2}
u(x)\leq \zeta(y)+L \tilde{d}(y, z)+M\tilde{d}(x, z)+M\vep\leq \zeta(y)+\max\{L, M\}\tilde{d}(x, y)+2\max\{L, M\}\vep.
\eeq
Since the above holds for all $\vep>0$,  we have \eqref{sol bdry regularity weak} immediately. 

In order to show \eqref{sol bdry regularity}, we need the stronger condition \eqref{bdry regularity}, which 
means that $\zeta$ is $m$-Lipschitz on $\partial\Omega$ with respect to $\tilde{d}$.
By \eqref{eq optimal control}, for any $\vep>0$, there exists $y_\vep\in \pO$ and 
a curve $\xi_\vep\in \A_x(\R, \Oba)$ such that with $t_\vep>0$ chosen so that $\xi_\vep(t_\vep)=y_\vep$, we have
\[
u(x)\geq \zeta(y_\vep)+\int_0^{t_\vep} f(\xi(s))\, ds-\vep\geq \zeta(y_\vep)+m \tilde{d}(x, y_\vep)-\vep.
\]
Using \eqref{bdry regularity},  we have 
\[
u(x)\geq \zeta(y)-m\tilde{d}(y, y_\vep)+m\tilde{d}(x, y_\vep)-\vep\geq \zeta(y)-m\tilde{d}(x, y)-\vep. 
\]
Letting $\vep\to 0$, we obtain 
\[
u(x)-\zeta(y)\geq -m\tilde{d}(x, y)\geq -M\tilde{d}(x, y),
\]
which, combined with \eqref{sol bdry regularity weak} with $L=m$, completes the proof. 
\end{proof}

The condition \eqref{bdry regularity} gives a quite restrictive constraint on the oscillation of the boundary data 
$\zeta$. A weaker condition than \eqref{bdry regularity} is that
\beq\label{bdry regularity2}
\left\{\begin{aligned}
&\zeta(x)-\zeta(y)\leq \int_0^t f(\xi(s))|\xi'|(s)\, ds\\
&\text{for every $x, y\in \pO$ and every rectifiable curve $\xi: [0, t]\to \Oba$ with $\xi(0)=x$ and $\xi(t)=y$.}
\end{aligned}
\right.
\eeq
This condition is employed in \cite{GHN} to guarantee the existence of continuous solutions to the 
Dirichlet  problem. We can use \eqref{bdry regularity2} instead of \eqref{bdry regularity} to obtain a 
continuity result weaker than \eqref{sol bdry regularity} provided that $\Omega$ 
enjoys a better regularity like the so-called quasiconvexity.

\begin{prop}[Boundary consistency under domain quasiconvexity]\label{prop bdry regularity2}
Let $(\X, d)$ be a complete rectifiably connected metric space and $\tilde{d}$ be given by \eqref{int metric}.
We suppose that \eqref{eq ast} holds. Let $\Omega\subsetneq \X$ be an open set. Assume that $f\in C(\Oba)$ 
is bounded and  $f\geq 0$ in $\Oba$. 
Assume in addition that $\Oba$ is $\sigma_\Omega$--convex in $(\X, \tilde{d})$ with respect to a modulus of continuity 
$\sigma_\Omega$, i.e., for any $x, y\in \Oba$, there exist a rectifiable curve $\xi$ in $\Oba$ joining $x$ and $y$ and satisfying 
$\ell(\xi)\leq \sigma_\Omega(\tilde{d}(x, y))$.
Let $\zeta$ satisfy \eqref{bdry regularity2} and $u$ be given by \eqref{eq optimal control}. 
Then \eqref{new bdry regularity} holds with $\sigma(t)=2\sigma_{\Omega}(t)\sup_{\Oba} f$ for $t\geq 0$. 
\end{prop}

\begin{proof}
Let us still take $M=\sup_{\Oba} f$ for simplicity of notation.  Fix $x\in \Omega$ and $y\in \pO$.  
Using the same argument as in the proof of Proposition \ref{prop bdry regularity}, we can easily prove that
\[
u(x)- \zeta(y)\leq 2M\sigma_\Omega(\tilde{d}(x, y)).
\]
Indeed, since the quasiconvexity of $\Omega$ and \eqref{bdry regularity2} yield 
\[
|\zeta(y_1)-\zeta(y_2)|\leq M\sigma_\Omega(\tilde{d}(y_1, y_2))
\]
for any $y_1, y_2\in \pO$, we only need to respectively substitute the terms $m\tilde{d}(y, z)$ and 
$m\tilde{d}(x, y)$ in \eqref{use later1} and \eqref{use later2}  with $M\sigma_\Omega(\tilde{d}(y, z))$ 
and $M\sigma_\Omega(\tilde{d}(x, y))$ . 

Let us now show that 
\beq\label{bdry reg2}
u(x)- \zeta(y)\geq -\sigma(\tilde{d}(x, y)).
\eeq
In fact, for any $\vep>0$, we can use \eqref{eq optimal control} again to find $y_\vep\in \pO$ and an 
arc-length parametrized curve $\xi_1\in \mathcal{A}_x([0, t_1))$ with $t_1>0$ such that 
$\tilde{d}(x, y_\vep)\geq t_1-\vep$ and
\beq\label{bdry reg1}
u(x)\geq \zeta(y_\vep)+\int_0^{t_1} f(\xi_1(s))\, ds-\vep.
\eeq
Note that there exists another arc-length parametrized curve $\xi_2\in \mathcal{A}_x([0, t_2))$ 
such that $\xi_2(0)=x$, $\xi_2(t_2)=y$ and $\sigma_\Omega(\tilde{d}(x, y))\geq t_2$. 
We thus can join $\xi_1$ and $\xi_2$ by taking 
\[
\xi(s)=\begin{cases}
\xi_1(t_1-s) &\text{if $s\in [0, t_1]$,}\\
\xi_2(s-t_1) &\text{if $s\in (t_1, t_1+t_2]$}.
\end{cases}
\]
Adopting \eqref{bdry regularity}, we have 
\[
\zeta(y)\leq \zeta(y_\vep)+\int_0^{t_1+t_2} f(\xi(s))\, ds,
\]
which, combined with \eqref{bdry reg1}, implies that 
\[
\begin{aligned}
u(x)&\geq \zeta(y)+\int_0^{t_1} f(\xi_1(s))\, ds-\int_0^{t_1+t_2} f(\xi(s))\, ds-\vep=\zeta(y)-\int_0^{t_2} f(\xi_2(s))\, ds-\vep\\
&\geq \zeta(y)-Mt_2-\vep\geq \zeta(y)-M\sigma_\Omega(\tilde{d}(x, y))-\vep. 
\end{aligned}
\]
We conclude the proof of \eqref{bdry reg2} by letting $\vep\to 0$. 
\end{proof}

\subsection{Slope-based solutions in general metric spaces}\label{sec:generalization}

Motivated by our above results, we generalize the definition of viscosity solutions proposed in \cite{GaS}. 
In particular, by utilizing the induced intrinsic metric $\tilde{d}$ given by \eqref{int metric}, we 
can now study the general equation \eqref{stationary eq} in a general metric space $(\X, d)$ 
that is not necessarily a length space but a complete rectifiably connected metric space satisfying \eqref{eq ast}. 

We first extend the notion of pointwise local slope for any locally Lipschitz function $u$ by continuing 
to use the notation $|\nabla u|$:
\beq\label{relaxed least slope}
|\nabla u|(x):=\limsup_{y\to x} {|u(y)-u(x)|\over \tilde{d}(x, y)},\quad \text{for any $x\in \X$.}
\eeq
Such an idea has already been formally stated in \cite[Equation (2.3)]{GHN}. 
If $(X, d)$ is a length space, then $\tilde{d}=d$ and \eqref{relaxed least slope} agrees with \eqref{slope}. 

Analogously, if $(\X, d)$ is a complete rectifiably connected metric space satisfying \eqref{eq ast},  
we can define upper and lower slopes by taking, for $u\in \lipl(\X)$,
\beq\label{half least slope}
|\nabla^{\pm} u|(x):=\limsup_{y\to x} {[u(y)-u(x)]_{\pm}\over \tilde{d}(x, y)},\quad \text{for any $x\in \X$.}
\eeq
These are again consistent with \eqref{semi slope} if $(X, d)$ is a length space. 

By using the definition in \eqref{test class},  we can provide the test classes $\ol{\mathcal{C}}(\Omega)$ 
and $\ul{\mathcal{C}}(\Omega)$ but on an open set $\Omega$ of a more general metric space.

As a result, $u\in \ol{\mathcal{C}}(\Omega)$ (resp., $u\in\ul{\mathcal{C}}(\Omega))$ if and only if for any $x\in \Omega$,
\[
\limsup_{y\to x}{u(y)-u(x)\over \tilde{d}(x, y)}
= |\nabla u|(x) \quad \left(\text{resp., } \limsup_{y\to x}{u(x)-u(y)\over \tilde{d}(x, y)}= |\nabla u|(x) \right).
\]
Now Definition \ref{defi s} can be used to define viscosity solutions of \eqref{stationary eq} in a complete 
rectifiably connected metric space as long as we replace $d$ by $\tilde{d}$.

We conclude this section by remarking that our approach above is based on a pointwise version 
of the notion of upper gradients, which, as an important substitute of the Euclidean gradients, 
has recently attracted a great deal of attention in the study of Sobolev spaces in metric measure 
spaces; see for instance \cite{Haj1, HKSTBook} for introduction on this topic.

\section{Monge Solutions and Local Equivalence}\label{sec:monge}

In this section we aim to show the equivalence of c- and s-solutions of the eikonal 
equation without relying on the boundary condition. Our discussion involves a third 
notion of solutions to  Hamilton-Jacobi equations in general metric spaces, which 
generalizes the so-called Monge solution of the eikonal equation in the Euclidean 
space studied in \cite{NeSu, BrDa} etc.

Thanks to our remarks in Section~\ref{sec:generalization}, it is sufficient to set up the 
problem in a complete length space $(\X, d)$. Our results in this section can be applied to  
more general rectifiably connected spaces by taking the induced intrinsic metric 
$\tilde{d}$ as in \eqref{int metric}. Throughout this section, we shall always focus our 
attention on the complete length space $(\X, d)$.

\subsection{Definition and uniqueness of Monge solutions}

Let us begin with the definition of Monge solutions  for the general Hamilton-Jacobi equation
in a complete length space. 
 
\begin{defi}[Definition of Monge solutions]\label{defi monge}
A function $u\in \lipl(\Omega)$ is called a {Monge subsolution} (resp., {Monge supersolution}) 
of \eqref{stationary eq} if, at any $x\in \Omega$, 
\[
H\left(x, u(x), |\nabla^- u|(x)\right)\leq 0 \quad \left(\text{resp.}, H\left(x, u(x), |\nabla^- u|(x)\right)\geq 0\right).
\]
A function $u\in \lipl(\Omega)$ is said to be a Monge solution if $u$ is both a Monge 
subsolution and a Monge supersolution, i.e., $u$ satisfies
\beq\label{stationary monge}
H\left(x, u(x), |\nabla^- u|(x)\right)=0
\eeq
at any $x\in \Omega$.  
\end{defi}

In the case of \eqref{eikonal eq}, the definition of Monge subsolutions (resp.,  supersolutions) reduces to 
\begin{equation}\label{eq:Monge-supsubsol-Eik}
|\nabla^- u|(x)\leq f(x)\quad (\text{resp., } |\nabla^- u|(x)\geq f(x))
\end{equation}
for all $x\in \Omega$.  Then $u$ is a Monge solution of \eqref{eikonal eq} if 
\begin{equation}\label{eq:Monge-sol-Eik}
|\nabla^- u|= f\quad \text{in $\Omega$}. 
\end{equation}

The notion of Monge solutions of \eqref{eikonal eq} in the Euclidean space is studied in \cite{NeSu}, where the 
right hand side $f$ is allowed to be more generally lower semicontinuous in $\overline{\Omega}$. 
Such a notion, still in the Euclidean space, was later generalized in \cite{BrDa} to handle general Hamilton-Jacobi 
equations with discontinuities. The definitions of Monge solutions in \cite{NeSu, BrDa} require the optical length 
function, which can be regarded as a Lagrangian structure for $H$.
In contrast, our definition of Monge solutions does not rely on optical length functions.
We only consider continuous $H$ in this work and the discontinuous case will be discussed in 
our forthcoming paper \cite{CLShZ}.

One advantage of using the Monge solutions is that its uniqueness can be easily obtained. In what follows, 
we give a comparison principle for Monge solutions of the eikonal equation. 

\begin{thm}[Comparison principle for Monge solutions of eikonal equation]\label{thm comparison monge}
Let $(\X, d)$ be a complete length space and $\Omega\subsetneq \X$ be a bounded open set in $(\X, d)$.  Assume that 
$f\in C(\Omega)$ is bounded and satisfies $\inf_{\Omega}f>0$.  
Let $u \in C(\Oba)\cap \lipl(\Omega)$ be a bounded Monge subsolution and $v \in C(\Oba)\cap \lipl(\Omega)$ be 
a bounded Monge supersolution of \eqref{eikonal eq}. If 
\beq\label{bdry verify monge}
\lim_{\delta\to 0}\sup\left\{u(x)-v(x): x\in \Oba, \ d(x, \pO)\leq \delta \right\}\leq 0, 
\eeq
then $u\le v$ in $\Oba$. Here $d(x, \pO)$ is given by $\inf_{y\in \pO} d(x, y)$.
\end{thm}

\begin{proof}
Since $u$ and $v$ are bounded, we may assume that $u, v\geq 0$ by adding a 
positive constant to them. It suffices to show that $\lambda u\le v$ in $\Omega$ for all $\lambda\in (0,1)$. 
Assume by contradiction that 
there exists $\lambda\in (0,1)$ such that $\sup_{\Omega}(\lambda u-v)> 2\mu$ for some $\mu>0$.  
By \eqref{bdry verify monge}, we may take $\delta>0$ small such that 
\[
\lambda u(x)-v(x)\leq u(x)-v(x)\leq \mu
\]
for all $x\in \Oba\setminus \Omega_\delta$, where we denote $\Omega_r=\{x:\Omega: d(x, \pO)>r\}$ for $r>0$. 
We choose $\vep\in (0, \delta/2)$ such that
\[
\sup_\Omega (\lambda u-v)>2\mu+\vep^2
\]
and
\beq\label{eps small}
\vep<(1-\lambda) \inf_{\Omega_{\delta/2}} f.
\eeq
We have such an $\vep>0$ because $\inf_{\Omega}f>0$.
Thus there exists $x_0\in \Omega$ such that $\lambda u(x_0)-v(x_0)\geq \sup_\Omega (\lambda u-v)-\vep^2>2\mu$ 
and therefore $x_0\in \Omega_\delta$. 

By Ekeland's variational principle (cf. \cite[Theorem 1.1]{Ek1}, \cite[Theorem 1]{Ek2}), there exists 
$x_\vep\in B_\vep(x_0)\subset \Omega_{\delta/2}$ such that 
\[
\lambda u(x_\vep)-v(x_\vep)\geq \lambda u(x_0)-v(x_0) 
\]
and $x\mapsto \lambda u(x)-v(x)-\vep d(x_\vep, x)$ attains a local maximum in $\Omega$ at $x=x_\vep$.
It follows that
\beq\label{comparison monge1}
v(x_\vep)-v(x)\le \lambda u(x_\vep)-\lambda u(x)+\vep d(x_\vep, x)
\eeq
for all $x\in B_{r}(x_\vep)$ when $r>0$ is small enough. 
Since $v$ is a Monge supersolution of \eqref{eikonal eq} and hence
satisfies~\eqref{eq:Monge-supsubsol-Eik}, there exists 
a sequence $\{y_n\}\subset \Omega$ such that 
\[
\lim_{y_n\to x_\vep}\frac{[v(y_n)-v(x_\vep)]_-}{d(x_\vep, y_n)}\geq f(x_\vep)>0.
\]
Note that here it is crucial to have $f(x_\vep)>0$ so that for large integers $n$ we have
$v(y_n)<v(x_\vep)$. 
Hence, by \eqref{comparison monge1} we have
\[
\begin{aligned}
f(x_\vep)& \leq \lim_{y_n\to x_\vep}\frac{\lambda [u(y_n)-u(x_\vep)]_-}{d(y_n, x_\vep)}+\vep\\
&\leq \limsup_{x\to x_\vep}\frac{\lambda [u(x)-u(x_\vep)]_-}{d(x, x_\vep)}+\vep\\
&=\lambda |\nabla^- u|(x_\vep)+\vep
\end{aligned}
\]
Using the fact that $u$ is a Monge subsolution, we get 
\[
f(x_\vep)\leq \lambda f(x_\vep)+\vep,
\]
which contradicts the choice of $\vep>0$ as in \eqref{eps small}. Our proof is thus complete.
\end{proof}

In the above theorem we cannot replace $|\nabla^-u|$ with $|\nabla u|$. 
A simple counterexample with $|\nabla u|$ is as follows: 
in $[-1,1]\subset\R$, both $1-|x|$ and $|x|-1$ would be two different solutions of \eqref{eikonal eq} with $f\equiv 1$ that satisfy the  
boundary condition $u(\pm 1)=0$, but only the former is the unique Monge solution.

\begin{rmk}\label{rmk uniqueness discontinuous}
It can be easily seen that the continuity assumption on $f$ is not utilized at all  in the proof above. Hence the 
comparison principle in Theorem \ref{thm comparison monge} still holds even if we drop the continuity 
assumption for $f$. In \cite{CLShZ} we study in detail Monge solutions for discontinuous 
Hamiltonians in general metrics measure spaces. 
\end{rmk}

\begin{rmk}
An analogous comparison principle can be established for locally Lipschitz Monge solutions 
of \eqref{stationary eq} provided that $p\mapsto H(x, \rho, p)$ is continuous in $[0, \infty)$ 
uniformly for all $(x, \rho)\in \Omega\times \R$ and $\rho\mapsto H(x, \rho, p)$ is strictly increasing in the sense that
 there exists $\mu>0$ such that 
\[
\rho\mapsto H(x, \rho, p)-\mu \rho
\]
is nondecreasing for all $x\in \Omega$ and $p\geq 0$. The proof is similar to that of Theorem \ref{thm comparison monge}.
\end{rmk} 

Concerning the existence of Monge solutions, Theorem \ref{thm equiv Monge} shows that any c-solution 
of \eqref{eikonal eq} is a Monge solution. In particular, the formula \eqref{eq optimal control} provides a 
unique c- and Monge solution of \eqref{eikonal eq} and \eqref{bdry cond} if \eqref{bdry regularity} holds.

\subsection{Local equivalence of solutions of eikonal equations}

We study the local relation between Monge solutions, c-solutions and s-solutions. 
Let us first discuss the subsolution properties.

\begin{prop}[Relation between c- and Monge subsolutions]\label{prop eikonal sub1}
Let $(\X, d)$ be a complete length space and $\Omega$ be an open set in $\X$. Assume 
that $f\in C(\Omega)$ with $f\geq 0$.  Let $u\in C(\Omega)$. Then 
$u$ is a c-subsolution of \eqref{eikonal eq} if and only if it is a Monge subsolution of \eqref{eikonal eq}. 
\end{prop} 
\begin{proof}
We begin with a proof of the implication ``$\Rightarrow$''. 
Since the local Lipschitz continuity of c-subsolutions $u$ is provided in Lemma~\ref{lem c-lip}, it suffices to 
verify that $|\nabla^- u|(x_0)\leq f(x_0)$ for every $x_0\in \Omega$. 
To this end, we fix $x_0\in\Omega$.
Since $u$ is a c-subsolution of \eqref{eikonal eq}, 
using \eqref{prop c-sub eq} with $s=0$ and $\xi(0)=x_0$ we have
\[  
u(x_0)-u(\xi(t))\leq \int_0^t f(\xi(s))\, ds
\] 
for any $\xi\in \mathcal{A}_{x_0}(\R, B_r(x_0))$ and $0\le t\le T^+_{B_r(x_0)}[\xi]$. Therefore  
\[
{u(x_0)-u(x)\over  d(x, x_0)}\leq {1\over  d(x, x_0)} \int_0^t f(\xi(s))\, ds \leq \frac{1}{d(x,x_0)} \ell(\xi) \sup_{B_r(x_0)}f
\]
for any $x\in B_r(x_0)$ and any curve $\xi\subset B_r(x_0)$ joining $x_0$ and $x$. By the continuity of $f$ 
and by the fact that $\X$ is a length space, taking the infimum over all 
$\xi$ and then sending $r\to 0$ we obtain 
\[
|\nabla^- u|(x_0)\leq f(x_0)
\]
as desired. 

We next prove the reverse implication ``$\Leftarrow$''. We again fix $x_0\in \Omega$ arbitrarily.  We take an 
arbitrary curve $\xi\in \A_{x_0}(\R, \Omega)$; in particular $\xi(0)=x_0$. Suppose that there is a function 
$\phi\in C^1(\R)$ such that $t\mapsto u(\xi(t))-\phi(t)$ attains a local maximum at $t=0$. Then there is some $t_0>0$
such that we have
\[
\phi(t)-\phi(0)\geq u(\xi(t))-u(\xi(0))
\]
when $-t_0<t<t_0$ . If $\phi'(0)=0$, then we obtain immediately the desired inequality 
\eqref{eq c-sub}. If $\phi'(0)\neq 0$, then without loss we may assume that 
$\phi'(0)<0$, in which case
$\phi(t)-\phi(0)<0$ for all $t\in (0, t_1)$ for sufficiently small $t_1\in (0, t_0)$. (If 
$\phi'(0)>0$, then we can consider $t\in (-t_1,0)$ instead below.) It follows that 
\[
|\phi'(0)|\leq \limsup_{t\to 0+}{u(\xi(0))-u(\xi(t))\over t}\leq |\nabla^- u|(x_0).
\]
Since $u$ is a Monge subsolution, we have $|\nabla^- u|(x_0)\leq f(x_0)$ and thus deduce \eqref{eq c-sub} again. 
\end{proof}

\begin{prop}[Relation between s- and Monge subsolutions]\label{prop eikonal sub2}
Let $(\X, d)$ be a complete length space and $\Omega$ be an open set in $\X$. Assume that 
$f$ is locally uniformly continuous and $f\geq 0$ in $\Omega$.
Let $u\in C(\Omega)$. Then the following results hold.
\begin{enumerate}
\item[(i)] If $u$ is a Monge subsolution of \eqref{eikonal eq}, then it is an s-subsolution of \eqref{eikonal eq}. 
\item[(ii)] If $u$ is a locally uniformly continuous s-subsolution of \eqref{eikonal eq}, then it is a Monge subsolution of \eqref{eikonal eq}. 
\end{enumerate}
\end{prop}
\begin{proof}
(i)  It is an immediate consequence of Proposition \ref{prop sub} and Proposition \ref{prop  eikonal sub1}. 

(ii) Take $\delta>0$ arbitrarily and $f_\delta=f+\delta$ in $\Omega$.  
Fix $x_0\in \Omega$ and $r>0$ small such that $B_{4r}(x_0)\subset \Omega$ and $u, f$ are both bounded
on $B_{4r}(x_0)$.  For $s\in [0, 4r)$ we set
\[
M_s:=\sup_{B_s(x_0)} f_\delta,
\]
and choose $M>0$ such that
\[
M>\max\left\{M_{2r}, {\sup_{B_{3r}(x_0)} u- u(x_0)\over r}\right\}. 
\]
Define a continuous function $g_r: [0, 3r)\to [0, \infty)$ by
\[
g_r(t):=M_r t+M(t-r)_+
\]
for $t\in [0, 3r)$. Due to the choice of $M$ above, the function defined by
\[
v_r(x)=u(x_0)+g_r(d(x_0, x))
\] 
for $x\in B_{3r}(x_0)$, satisfies $v_r\geq u$ on $\partial B_{2r}(x_0)$.
Moreover, for any $x\in B_{2r}(x_0)$ with $x\neq x_0$ and any $\vep>0$, we can find an arc-length 
parametrized curve $\xi$ such that $\xi(0)=x$, $\xi(t_\vep)=x_0$ 
and $t_\vep-\vep\leq d(x, x_0)\leq t_\vep$. For any $t\in [0, t_\vep]$,  
we have
\[
d(x_0,\xi(t))\ge d(x_0,x)-d(x,\xi(t))\ge (t_\vep-\vep)-(t_\vep-t)=t-\vep.
\]
Therefore
\[
t-\vep \leq d(x_0, \xi(t))\le \ell(\xi[0,t])\leq d(x, x_0)+\vep-d(x, \xi(t)).
\]
Taking $x_\vep=\xi(\sqrt{\vep})$, we have 
\[
{v_r(x)-v_r(x_\vep)\over d(x, x_\vep)}
 \geq \frac{g_r(d(x_0, x))-g_r(d(x_0, x_\vep))}{d(x, x_0)-d(x_0, x_\vep)+\vep}\geq {1\over 1+2\sqrt{\vep}} 
 \frac{g_r(d(x_0, x))-g_r(d(x_0, x_\vep))}{d(x, x_0)-d(x_0, x_\vep)}
\] 
when $\vep>0$ is sufficiently small. 
Hence if $0<d(x,x_0)\le  r$,  then for sufficiently small $\vep>0$ we have $g_r(d(x_\vep,x_0))\ge M_rd(x_\vep,x_0)$,
and then by the choice of $M_r$ we have
\[
\limsup_{x_\vep\to x} \frac{g_r(d(x_0, x))-g_r(d(x_0, x_\vep))}{d(x, x_0)-d(x_0, x_\vep)}= M_r\ge f_\delta(x).
\]
If $r< d(x, x_0)<2r$, then for sufficently small $\vep$ we have $d(x_\vep,x_0)>r$ as well, and so we get
\[
\frac{g_r(d(x_0, x))-g_r(d(x_0, x_\vep))}{d(x, x_0)-d(x_0, x_\vep)}
=\frac{(M_r+M)d(x_0,x)-(M_r+M)d(x_0,x_\vep)}{d(x, x_0)-d(x_0, x_\vep)}.
\]
Therefore as $d(x_0,x)<2r$, we have 
\[
\limsup_{x_\vep\to x} \frac{g_r(d(x_0, x))-g_r(d(x_0, x_\vep))}{d(x, x_0)-d(x_0, x_\vep)}\geq M> M_{2r}\ge f_\delta(x).
\]
In either case, we see that $|\nabla^- v_r|(x)\geq f_\delta(x)$; in other words, $v_r$ is a Monge supersolution of 
\[
|\nabla u|=f_\delta \quad \text{in $B_{2r}(x_0)\setminus \{x_0\}$}. 
\]
In view of Proposition \ref{prop eikonal super}(i), we see that $v_r$ is an s-supersolution of the same equation
(keep in mind also that by its construction, $v_r$ is $M$-Lipschitz); in particular, we have
\beq\label{bdry verify1}
v_r(x)-u(y)\geq v_r(x)-v_r(y)\geq -Md(x, y)
\eeq
for all $x\in B_{2r}(x_0)$ and $y\in \partial B_{2r}(x_0)\cup \{x_0\}$. 

On the other hand, $u$ is an s-subsolution and is uniformly continuous in $B_{2r}(x_0)$ with some modulus $\sigma_0$. We have
\[
u(x)-u(y)\leq \sigma_0(d(x, y))
\]
for all $x\in B_{2r}(x_0)$ and $y\in \partial B_{2r}(x_0)\cup \{x_0\}$. Combining this with \eqref{bdry verify1}, 
we have shown that the condition \eqref{bdry verify0} holds with $\Omega=B_{2r}(x_0)\setminus \{x_0\}$, 
$v=v_r$, $\zeta=u$ and $\sigma(s)=\max\{Ms, \sigma_0(s)\}$ for $s\geq 0$. 
Since $f_\delta=f+ \delta$ in $\Omega$, the function $u$ must also be an s-subsolution for the eikonal equation 
related to the function $f_\delta$. We thus can use the comparison result 
\cite[Theorem~5.3]{GaS} to get $u\leq v_r$ in $B_{2r}(x_0)\setminus \{x_0\}$. Letting $\delta\to 0$, we are led to
\[
u(x)\leq u(x_0)+d(x, x_0)\sup_{B_r(x_0)} f\quad \text{for all $x\in B_r(x_0)$. }
\]
One can use the same argument to show that for all $x, y\in B_{r/4}(x_0)$ (and therefore $d(x, y)\leq r/2$),
\[
u(y)\leq u(x)+d(x, y)\sup_{B_{r/2}(x)} f\leq  u(x)+d(x, y) \sup_{B_r(x_0)} f,
\]
which yields (recalling that we chose $r>0$ small enough so that $f$ is bounded on $B_{4r}(x_0)$)
\[
|u(x)-u(y)|\leq d(x, y)\sup_{B_r(x_0)} f. 
\]
This immediately implies that 
\beq\label{sub char}
|\nabla^- u|(x_0)\leq |\nabla u|(x_0)\leq f(x_0).
\eeq
Hence, we can conclude that $u$ is a Monge subsolution of \eqref{eikonal eq}, since $x_0$ is arbitrarily taken. 
\end{proof}

In the proof of (ii) above, the local uniform continuity of $u$ and $f$ (especially near $x_0$) enables us to adopt 
the comparison principle. The uniform continuity can be removed if the space $(\X, d)$ has some compactness a priori. 

We next turn to the relation between supersolutions. 

\begin{prop}[Relation between supersolutions]\label{prop eikonal super}
Let $(\X,  d)$ be a complete length space and $\Omega$ be an open set in $\X$. 
Assume that $f$ is locally uniformly continuous 
and $f\ge 0$ in $\Omega$.  Let $u\in \lipl(\Omega)$. Then 
\begin{enumerate}
\item[(i)] $u$ is a Monge supersolution of \eqref{eikonal eq} if and only if $u$ is an s-supersolution of \eqref{eikonal eq}. 
\item[(ii)]  Assume in addition that $f>0$ on $\Omega$. If $u$ is a local c-supersolution of \eqref{eikonal eq}, then 
$u$ is a Monge supersolution of \eqref{eikonal eq}. 
\end{enumerate}
\end{prop}
\begin{proof}
(i) Let us first show the equivalence between a Monge supersolution and an s-supersolution. 
We begin with the implication ``$\Rightarrow$''. Let $u$ be a Monge supersolution. 
Suppose that there exist 
 $\psi_1\in \ol{\mathcal{C}}(\Omega)$ and $\psi_2\in \lipl(\Omega)$ such that $u-\psi_1-\psi_2$ attains a local minimum at $x_0$. 
 If $f(x_0)=0$, then the desired inequality
\[
 |\nabla \psi_1|(x_0)\geq -|\nabla \psi_2|^\ast(x_0)
\]
is trivial. It thus suffices to consider the case $f(x_0)>0$. By the definition of Monge supersolutions, for any $x_0\in \Omega$, 
\[
|\nabla^- u|(x_0)\geq f(x_0)>0. 
\]
Then for any $\delta>0$, for each $\vep>0$ we can find $x_\vep\in \Omega$ with $x_\vep\to x_0$ as $\vep\to 0$ such that
\beq\label{eq weak1}
u(x_\vep)-u(x_0)\leq (-f(x_0)+\delta) d(x_0, x_\vep). 
\eeq
It follows from \eqref{eq weak1} and the maximality of $u-\psi_1-\psi_2$ at $x_0$ that 
\[
\psi_1(x_\vep)-\psi_1(x_0)+\psi_2(x_\vep)-\psi_2(x_0)\leq (-f(x_0)+\delta) d(x_0, x_\vep), 
\]
which implies that 
\[
{|\psi_1(x_\vep)-\psi_1(x_0)|\over  d(x_0, x_\vep)}
\ge \frac{\psi_1(x_0)-\psi_1(x_\vep)}{d(x_0,x_\vep)}
\geq f(x_0)-\delta-{|\psi_2(x_\vep)-\psi_2(x_0)|\over  d(x_0, x_\vep)}.
\]
Letting $\vep\to 0$ and then $\delta\to 0$, we obtain
\[
|\nabla \psi_1|(x_0)\geq f(x_0)-|\nabla \psi_2|^\ast(x_0).
\]
It follows that $u$ is an s-supersolution.  

The proof for the reverse implication ``$\Leftarrow$'' is given next. So we assume that $u$ is an
s-supersolution. Suppose that $u$ is not a Monge supersolution. Then
there exists $x_0\in \Omega$ such that 
\[
\limsup_{x\to x_0}\frac{[u(x_0)-u(x)]_+}{ d(x,x_0)}<f(x_0).
\]
A contradiction is immediately obtained if $f(x_0)=0$. We thus only consider $f(x_0)>0$ below.
Then there exists $\delta>0$ such that for all $x\in B_\delta(x_0)$, 
\beq\label{eq monge-s1}
\frac{u(x_0)-u(x)}{ d(x,x_0)}-f(x_0)\leq -2\delta.
\eeq
By the continuity of $f$, for any $0<\vep<\min\{\delta,f(x_0)/2\}$, we can choose $0<r<\delta$ such that 
\beq\label{continuity f}
|f(x)-f(x_0)|\leq \vep \quad \text{for all $x\in B_r(x_0)$.}
\eeq
Observe that $f(x)\ge f(x_0)/2>0$ whenever $x\in B_r(x_0)$.
Let
\[
v(x):=u(x_0)-(f(x_0)-\vep) d(x,x_0)+\delta r.
\]
Then in view of \eqref{eq monge-s1}, we have $v(x)\le u(x)$ for all $x\in \partial B_r(x_0)$.  
Moreover,  we claim that $v$ is an s-subsolution of 
 \[
 |\nabla v|=f \quad \text{in $B_r(x_0)$}.
 \]
Indeed, for any $x\in B_r(x_0)$, if $v-\psi_1-\psi_2$ achieves a maximum at $x$, where 
$\psi_1\in \underline{\mathcal{C}}(\Omega)$ and $\psi_2\in \text{Lip}_{loc}(\Omega)$, then
\[
\psi_1(x)-\psi_1(y)\le v(x)-v(y)+\psi_2(y)-\psi_2(x).
\]
It follows that 
\[
\begin{aligned}
|\nabla \psi_1|(x)=\limsup_{y\to x}\frac{\psi_1(x)-\psi_1(y)}{d(x, y)} &\le \limsup_{y\to x}\frac{v(x)-v(y)}{d(x, y)}+\limsup_{y\to x}\frac{\psi_2(y)-\psi_2(x)}{d(x, y)}\\
&\le f(x_0)-\vep+|\nabla \psi_2|^*(x)\\
&\le f(x)+|\nabla \psi_2|^*(x).
\end{aligned}
\]
The claim has been proved. 
Applying the comparison principle for s-solutions, we have $v\le u$ in $B_r(x_0)$. This contradicts the fact that 
\[
v(x_0)=u(x_0)+\delta r>u(x_0).
\]
Our proof for the equivalence of supersolutions is now complete.

(ii)  It follows immediately from (i) and Remark \ref{local c to s}.
Note that Remark \ref{local c to s} requires $\inf_{B_r(x)}f>0$ for any $x\in \Omega$ and $r>0$ small, 
which is implied by the continuity and positivity of $f$ in $\Omega$. 
\end{proof}

We are not able to show that any Monge supersolution of \eqref{eikonal eq} is a local c-supersolution. 
However, if $u$ is a Monge solution, then we can show that $u$ must be a local c-solution. 

\begin{prop}[Relation between Monge and local c-solutions]\label{prop eikonal solution}
Let $(\X,  d)$ be a complete length space and $\Omega$ be an open set in $\X$. 
Assume that $f$ is locally uniformly continuous 
and $f>0$ in $\Omega$.  If $u$ is a Monge solution of \eqref{eikonal eq}, then $u$ is a local c-solution of \eqref{eikonal eq}.
\end{prop}

\begin{proof}
Suppose that $u$ is a Monge solution of \eqref{eikonal eq} in $\Omega$. By 
Proposition~\ref{prop eikonal sub1}, we know that $u$ must be a c-subsolution. 
It suffices to show that $u$ is a local c-supersolution of \eqref{eikonal eq}. 

 For any $x_0\in \Omega$, take $r>0$ small such that $u$ is Lipschitz in $\overline{B_r(x_0)}$, that is, there exists $L>0$ such that
\beq\label{lip local solution}
|u(x)-u(y)|\leq Ld(x, y) \quad\text{for any $x, y\in \overline{B_r(x_0)}$.}
\eeq
Letting  $\zeta(y)=u(y)$ for all $y\in \partial B_r(x_0)$, by \eqref{eq optimal control} with 
$\Omega=B_r(x_0)$ we have the unique c-solution $U$ in $B_r(x_0)$ given by
\begin{equation}\label{local formula}
U(x):=\inf\bigg\{ \int_{0}^{t_r^+} f(\xi(s))\, ds+ u\left(\xi(t_r^+)\right)\, :\, \xi\in \A_{x}(\R, \X) \text{ with }0<T^+_{B_r(x_0)}[\xi]<\infty\bigg\}.
\end{equation}
It follows from Proposition \ref{prop eikonal sub1} and Proposition \ref{prop eikonal super}(ii) that $U$ is a 
Monge solution of the eikonal equation in $B_r(x_0)$. 
Note also that, by Proposition \ref{prop bdry regularity}(1), 
\[
U(x)-u(y)\leq d(x, y )\max\left\{L,\ \sup_{B_r(x_0)} f\right\} \quad \text{for any $x\in B_r(x_0)$ and $y\in \partial B_r(x_0)$.}
\]
We then can adopt the comparison principle,  Theorem \ref{thm comparison monge}, to get 
$U\leq u$ in $B_r(x_0)$.
In view of \eqref{local formula}, it follows that for any $x\in B_r(x_0)$ and any $\vep>0$ small, there 
exists a curve $\xi_\vep\in A_x(\R, \X)$ such that
\beq\label{local solution1}
u(x)\geq U(x)\geq \int_{0}^{t_r^+} f(\xi_\vep(s))\, ds+ u\left(\xi_\vep(t_r^+)\right)-\vep,
\eeq
where $t_r^+=T^+_{B_r(x_0)}[\xi_\vep]$ denotes the exit time of $\xi_\vep$ from $B_r(x_0)$. 
On the other hand, since $u$ is a c-subsolution, we can use Proposition \ref{prop c-sub} to get, for any $0\leq t\leq t_r^+$,
\beq\label{local solution2}
u(\xi_\vep(t))\leq u\left(\xi_\vep(t_r^+)\right)+\int_t^{t_r^+} f(\xi_\vep(s))\, ds. 
\eeq
Combining \eqref{local solution1} and \eqref{local solution2}, we deduce that for any $0\leq t\leq t_r^+$,  
\[
u(x)\geq u(\xi_\vep(t))+\int_0^t f(\xi_\vep(s))\, ds-\vep.
\]
Setting 
\[
\xi(t)=\begin{cases}
\xi_\vep(t) & \text{if $t\geq 0$,}\\
\xi_\vep(-t) & \text{if $t<0$,}
\end{cases}
\quad\text{and }\quad w(t)=u(x)-\int_0^t f(\xi(s))\, ds
\]
for $t\in \R$, we easily see that $(\xi, w)$ satisfies the conditions for local c-supersolutions in 
Definition \ref{def local c}. Indeed, $w$ is of class
$C^1$ in $(-t_r^+, t_r^+)\setminus \{0\}$ with $w'=f\circ \xi$ in $(-t_r^+, t_r^+)\setminus \{0\}$. 
If there is $\phi\in C^1(-t_r^+, t_r^+)$ such that $w-\phi$ achieves a minimum at some 
$t_0\in (-t_r^+, t_r^+)$, then $t_0\neq 0$ since $f>0$ in $\Omega$. It then follows that $\phi'(t_0)=w'(t_0)$, which yields 
\[
|\phi^\prime|(t_0)=|w^\prime|(t_0)=f(\xi(t_0)).  
\]
Hence, $u$ is a local c-supersolution and therefore a local c-solution. 
\end{proof}

We now complete the proof of Theorem \ref{thm equiv Monge}.

\begin{proof}[Proof of Theorem \ref{thm equiv Monge}]
The proof consists of the results in Propositions \ref{prop eikonal sub1}, \ref{prop eikonal sub2}, 
\ref{prop eikonal super} and \ref{prop eikonal solution}. In addition, 
combining \eqref{sub char} and the definition of Monge supersolutions, we have \eqref{regular0} if any of (a), (b) and (c) holds. 
\end{proof}

\begin{rmk}
It is worth pointing out that, due to Proposition \ref{prop eikonal sub2} and Proposition \ref{prop eikonal super}(i), 
the equivalence between Monge solutions and locally uniformly continuous s-solutions still holds even if the 
assumption  $f>0$ is relaxed to $f\geq 0$ in $\Omega$. See Theorem \ref{thm equiv general} a
nd \ref{thm equiv general proper} below for a more general result focusing on the relation between Monge and s-solutions. 
\end{rmk}

The local uniform continuity of $f$ and $u$  
in Theorem~\ref{thm equiv Monge} can be dropped if the space $(\X, d)$ is assumed to proper, 
that is, any closed bounded subset of $\X$ is compact.

\begin{cor}[Local equivalence in a proper space]\label{cor equiv Monge}
Let $(\X,  d)$ be a proper complete 
geodesic space and $\Omega$ be a bounded open set in $\X$. Assume that $f\in C(\Omega)$ and $f>0$ in $\Omega$. 
Let $u\in C(\Omega)$. Then the following statements are equivalent: 
\begin{enumerate}
\item[(a)] $u$ is a c-solution of \eqref{eikonal eq};
\item[(b)] $u$ is a s-solution of \eqref{eikonal eq};
\item[(c)] $u$ is a Monge solution of \eqref{eikonal eq}. 
\end{enumerate}
In addition, if any of (a)--(c) holds, then $u\in \lipl(\Omega)$ and satisfies \eqref{regular0}. 
\end{cor}

\subsection{General Hamilton-Jacobi equations}

We next turn to a more general class of Hamilton-Jacobi equations. In this case, c-solutions are no longer defined. 
We can still show the equivalence between Monge solutions and s-solutions. 


\begin{thm}[Equivalence of Monge and $s$-solutions of general equations]\label{thm equiv general}
Let $(\X,  d)$ be a complete length space and $\Omega\subset \X$ be an open set.  Let 
$H: \Omega\times \R\times [0, \infty)\to \R$ be continuous and satisfy the following conditions:
\begin{enumerate}
\item[(1)] $(x, \rho)\mapsto H(x, \rho, p)$ is locally uniformly continuous, that is, for any $x_0\in \Omega$ 
and $\rho_0\in \R$, there exist $\delta>0$ small and a modulus of continuity $\omega$ such that 
\[
|H(x_1, \rho_1, p)-H(x_2, \rho_2, p)|\leq \omega\left(d(x_1, x_2)+|\rho_1-\rho_2|\right)
\]
for all $x_1, x_2\in B_\delta(x_0)$, $\rho_1, \rho_2\in [\rho_0-\delta, \rho_0+\delta]$ and $p\geq 0$. 
\item[(2)]  For any $x_0\in \Omega$ and $\rho_0\in \R$, there exist $\delta>0$ and  $\lambda_0>0$ such that 
\[
p\mapsto H(x, \rho, p)-\lambda_0 p
\] 
is increasing for every $(x, \rho)\in B_\delta(x_0)\times [\rho_0-\delta, \rho_0+\delta]$. 
\item[(3)] $p\mapsto H(x, \rho, p)$ is coercive in the sense that 
\beq\label{coercivity2}
\inf_{(x, \rho)\in \Omega\times [-R, R]} H(x, \rho, p)\to \infty \quad \text{as $p\to \infty$ for any $R>0$. }
\eeq
\end{enumerate}
 Then $u$ is a Monge solution of \eqref{stationary eq} if and only if $u$ is a locally uniformly continuous 
 s-solution of \eqref{stationary eq}. In addition, such $u$ is locally Lipschitz in $\Omega$. 
\end{thm}

\begin{proof}
Let $u$ be either a Monge solution or a locally uniformly continuous s-solution. 
We first claim that 
\beq\label{hamiltonian low bound}
H(x, u(x), 0)\leq 0\quad \text{for any $x\in \Omega$.}
 \eeq
Thanks to the condition (2), this is clearly true when $u$ is a Monge solution.  It thus suffices to 
show \eqref{hamiltonian low bound} for a locally uniformly continuous $s$-solution $u$. 
Fix $x_0\in \Omega$ arbitrarily. Let 
\[
\psi_1(x)={1\over \vep}d(x, x_0)^2
\]
for $\vep>0$ small. Then, due to the local boundedness of $u$, there exist $\delta>0$ and  
$y_\vep\in B_\delta(x_0)\subset \Omega$  such that 
\[
(u-\psi_1)(y_\vep)\geq \sup_{B_\delta(x_0)} (u-\psi_1)-\vep^2
\]
and $y_\vep\to x_0$ as $\vep\to 0$.
By Ekeland's variational principle (cf. \cite[Theorem 1.1]{Ek1}, \cite[Theorem 1]{Ek2}), there is 
a point $x_\vep\in B_\vep(y_\vep)$ such that 
\[
(u-\psi_1)(x_\vep)\geq (u-\psi_1)(y_\vep) 
\]
and $u-\psi_1-\psi_2$ attains a local maximum in $B_\delta(x_0)$ at $x_\vep\in B_\vep(y_\vep)$, where 
\[
\psi_2(x)=\vep d(x_\vep, x). 
\] 
It is clear that $x_\vep\to x_0$ as $\vep\to 0$. 
Since $u$ is an $s$-subsolution, we have 
\[
\inf_{|\rho|\leq \vep}H\left(x_\vep, u(x_\vep), {2\over \vep}d(x_\vep, x_0)+\rho\right)\leq 0,
\]
which, by the condition (2), yields 
\[
H\left(x_\vep, u(x_\vep), 0\right) \leq 0.
\]
Letting $\vep\to 0$, by the continuity of $H$, we deduce \eqref{hamiltonian low bound} at $x=x_0$. 
We have completed the proof of the claim. 

By the coercivity condition (3) we can define a 
function $h: \Omega\to [0, \infty)$ to be
\beq\label{implicit}
h(x):=\inf\{p\geq 0: H(x, u(x), p)> 0\}, 
\eeq 
and, thanks to the continuity of $H$ and the condition (2), we see that 
for each $x\in\Omega$, $h(x)\geq 0$ is the unique value satisfying
\begin{equation}\label{eq:Hamil-Eik}
H(x, u(x), h(x))=0. 
\end{equation}

We next claim that $h$ is locally uniformly continuous in $\Omega$. To see this, fix $x_0\in \Omega$ and an 
arbitrarily small $\delta>0$. We take $x, y\in B_{\delta_1}(x_0)$ with $\delta_1\leq \delta$ sufficiently small 
such that $u(x), u(y)\in [u(x_0)-\delta, u(x_0)+\delta]$. Then by the condition (1) we have 
\beq\label{local uniform ham}
H(x, u(x), p)-H(y, u(y), p)\leq \omega(d(x, y)+|u(x)-u(y)|)
\eeq
for any $p\geq 0$. 
Denote $W(x, y):=\omega(d(x, y)+|u(x)-u(y)|)$ for simplicity.
Since $H$ satisfies the condition (2), 
we can use \eqref{local uniform ham} to get, for any $x, y\in B_{\delta_1}(x_0)$,
\[
H\left(x, u(x), h(y)+{1\over \lambda_0} W(x, y) \right)
\geq H(x, u(x), h(y))+W(x, y)\geq H(y, u(y), h(y))=0,
\]
which, by \eqref{implicit} and \eqref{eq:Hamil-Eik}, yields 
\[
h(x)\leq h(y)+{1\over \lambda_0} W(x, y).
\]
We can analogously show that 
\[
h(x)\geq h(y)-{1\over \lambda_0} W(x, y)
\]
and therefore $h$ is uniformly continuous in $B_{\delta_1}(x_0)$.

From~\eqref{eq:Hamil-Eik} it is clear that $u$ is a Monge solution of~\eqref{stationary eq}  
if and only if $|\nabla^-u|(x)= h(x)$, that is, $u$ is a Monge solution of \eqref{eikonal eq} with $f=h$.
Note however that the function $h$ depends on $u$ implicitly in general.

We now show that $u$ is an s-solution of \eqref{stationary eq} if and only if $u$ is an s-solution of 
\eqref{eikonal eq} with $f=h$. 
To see this, suppose that there exist $x_0\in \Omega$, $\psi_1\in \ul{\mathcal{C}}(\Omega)$ and 
$\psi_2\in \text{Lip}_{loc}(\Omega)$ such that $u-\psi_1-\psi_2$ attains a maximum in $\Omega$ at $x_0$. 
Then by the monotonicity of $p\to H(x, \rho, p)$, the viscosity inequality \eqref{s-sub eq} holds
at $x_0$ if and only if 
\[
H\left(x_0, u(x_0), (|\nabla\psi_1|(x_0)-|\nabla \psi_2|^\ast(x_0))\vee 0\right)\leq 0,
\]
which, due to \eqref{implicit}, amounts to saying that
\[
|\nabla\psi_1|(x_0)-|\nabla \psi_2|^\ast(x_0)\leq h(x_0),
\]
that is, $u$ is an $s$-subsolution of \eqref{eikonal eq} with $f=h$. Analogous results for
supersolutions can be similarly proved. 

Noticing that Monge solutions and locally uniformly continuous s-solutions of \eqref{eikonal eq} have 
been proved to be (locally) equivalent in Theorem~\ref{thm equiv Monge}, we immediately obtain the 
equivalence of both notions for \eqref{stationary eq} and local Lipschitz continuity. 
\end{proof}

As in Corollary \ref{cor equiv Monge}, if $(\X, d)$ is additionally assumed to be proper, then in 
Theorem~\ref{thm equiv general} we can drop the assumption on the local uniform continuity of s-solutions. 
Moreover, when $(\X, d)$ is proper, in the proof above we only need to show the continuity of $h$ as 
in \eqref{implicit}, since continuity implies local uniform continuity.  
Thus the condition (1) can be removed and (2) can be weakened by merely assuming that for every 
$(x, \rho)\in \Omega\times \R$, the map $p\mapsto H(x, \rho, p)$ is strictly increasing 
in $(0, \infty)$. Below we state the result without proofs.

\begin{thm}[Equivalence of Monge and $s$-solutions in a proper space]\label{thm equiv general proper}
Let $(\X,  d)$ be a complete proper geodesic space and $\Omega\subset \X$ be an open set.  Let 
$H: \Omega\times \R\times [0, \infty)\to \R$ be continuous and 
be coercive as in \eqref{coercivity2}. 
Assume that for every $(x, \rho)\in \Omega\times \R$, $p\mapsto H(x, \rho, p)$ is strictly increasing  in $(0, \infty)$. 
 Then $u$ is a Monge solution of \eqref{stationary eq} if and only if $u$ is an
 s-solution of \eqref{stationary eq}. In addition, such $u$ is locally Lipschitz in $\Omega$. 
\end{thm}

The strict monotonicity of $p\to H(x, \rho, p)$ assumed in Theorem~\ref{thm equiv general} and 
Theorem~\eqref{thm equiv general proper}  enables us to apply an implicit function argument. 
Although it is not clear to us if one can weaken the requirement, 
the examples below show that the equivalence result fails to hold in general if $H$ is not monotone in $p$.

\begin{example}
Let $\Omega=\X=\R$ with the standard Euclidean metric. Let 
\[
H(p)=1-|p-2|+\max\{p-3, 0\}^2, \quad p\geq 0.
\]
One can easily verify that this Hamiltonian satisfies all assumptions in Theorem \ref{thm equiv general} except for the monotonicity. 

It is not difficult to see that the function $u$ given by $u(x)=-3|x|$ for $x\in \R$
is a Monge solution of
\beq\label{example hj}
H(|\nabla u|)=0 \quad \text{in $\R$},
\eeq
since $|\nabla^- u|=3$ in $\R$.
 It is however not a conventional viscosity solution or an s-solution, though it is an s-supersolution. 
\end{example}

\begin{example}
Let $\Omega=\X=\R$ again. Set 
\[
H(p)=1-|p|+\max\{p-3, 0\}^2, \quad p\geq 0,
\]
which again satisfies all assumptions in Theorem \ref{thm equiv general} but the monotonicity.
This time we have 
\[
u(x)=|x|, \quad x\in \R
\]
as a viscosity solution or s-solution of \eqref{example hj}. But it is not a Monge solution, since $|\nabla^- u|(0)=0$ and $H(0)\neq 0$. 
\end{example}

\subsection{Further regularity}\label{sec:regularity}

Motivated by the observation \eqref{regular0} in Theorem \ref{thm equiv Monge}, we consider 
a higher regularity than the Lipschitz continuity related to the Monge solutions of \eqref{stationary eq}. 
One can slightly strengthen Definition \ref{defi monge} by further requiring that the solution $u$  
satisfy \eqref{regular}. We say that a solution $u$ is regular in $\Omega$ if \eqref{regular} holds. The reason why we 
regard this condition as regularity will be clarified below.

The property~\eqref{regular}  is studied for time-dependent Hamilton-Jacobi equations on metric 
spaces \cite[Proof of Theorem~2.5 and Remark~2.27]{LoVi}. The authors of \cite{LoVi} show that at any fixed time, 
such spatial regularity holds for the Hopf-Lax formula almost everywhere if $\X$ is a metric measure space that 
satisfies a doubling condition and a local Poincar\'e inequality, or everywhere if $\X$ is a finite dimensional 
space with Alexandrov curvature bounded from below. 
In our stationary setting, we have shown that any Monge solution of the eikonal equation \eqref{eikonal eq} is 
regular in a complete length space without using \emph{any} measure structure or imposing any assumptions on 
the space dimension and curvature, see Theorem~\ref{thm equiv Monge}. Now we consider the more general
Hamiltonian $H$.

\begin{prop}[Regularity of Monge solutions]\label{prop regular general}
Suppose that the assumptions in Theorem~\ref{thm equiv general} or Theorem~\ref{thm equiv general proper} hold.
If $u$ is a Monge solution of \eqref{stationary eq}, then $u$ is regular in $\Omega$ and $u\in \ul{\mathcal{C}}(\Omega)$. 
\end{prop}

\begin{proof}
We have shown in the proof of Theorem \ref{thm equiv general} that $u$ is locally a Monge solution of 
\[
|\nabla u|=h(x),
\]
where $h\in C(\Omega)$ is given by \eqref{implicit} and we have $h\geq 0$ in $\Omega$. 
The regularity \eqref{regular} is thus an immediate consequence of \eqref{sub char} in the proof of Proposition~\ref{prop eikonal sub2} and the definition of Monge solutions. We also see that $u\in \ul{\mathcal{C}}(\Omega)$. 
One can also apply the same proof to obtain the conclusion under the assumptions of Theorem \ref{thm equiv general proper}.
\end{proof}

We emphasize that for \eqref{regular} the strict monotonicity of $p\to H(x, \rho, p)$ cannot be relaxed, as 
indicated by the following simple example.

\begin{example}\label{ex:non-regular}
Let 
\[
H(p)=\begin{cases}
p &\text{when $0\leq p< 1$,}\\
1 &\text{when $1\leq p<2$,}\\
p-1 &\text{when $p\geq 2$.}
\end{cases}
\]
In this case, it is easily seen that 
\[
u(x)=\begin{cases}
x &\text{for $x\leq 0$,}\\
2x & \text{for $x>0$}
\end{cases}
\]
is an s-solution (usual viscosity solution) of 
\[
H(|\nabla u|)=1\quad \text{in $\X=\R$},
\]
but $|\nabla u|(0)=2\neq 1=|\nabla^- u|(0)$. 
\end{example}

We finally remark that \eqref{regular} can be regarded as a type of regularity that is related 
to the semi-concavity in the Euclidean spaces. Let us below briefly recall several results on 
semi-concave functions in $\R^N$; we refer the reader to \cite{CaS} for detailed introduction 
to the classical results and to \cite{Alb, ACS} for recent developments in more general sub-Riemannian contexts.

Recall (cf. \cite[Definition 2.1.1]{CaS}) that for any open set $\Omega\subset \R^N$, 
$u\in C(\Omega)$ is said to be a (generalized) semi-concave function in $\Omega$ if there 
exists a nondecreasing upper semicontinuous function $\omega: [0, \infty)\to [0, \infty)$ such that 
\[
\lim_{r\to 0+}\omega(r)=0
\]
and
\beq\label{general semiconcave}
u(x+\eta)+u(x-\eta)-2u(x)\leq \omega(|\eta|)|\eta|
\eeq
for any $x\in \Omega$ and any $\eta\in \R^N$ with $|\eta|$ sufficiently small. A more well-known 
notion of semi-concave functions is to ask $u\in C(\Omega)$ to satisfy \eqref{general semiconcave} with 
$\omega(r)=cr$ for some $c>0$; in this case, $u$ is semi-concave in $\Omega$ if and only if 
$x\mapsto u(x)-c|x|^2/2$ is concave in $\Omega$. 

It is well known \cite[Theorem 5.3.7]{CaS} that viscosity solutions of \eqref{stationary eq} are 
locally semi-concave (in the generalized sense) provided that $H(x, \rho, p)$ is locally Lipschitz and 
strictly convex in $p$; one can apply this result to \eqref{eikonal eq} by changing the form into 
\[
|\nabla u|^2=f(x)^2 \quad\text{ in $\Omega$}
\]
in order to meet the requirement of strict convexity.  Moreover, any such semi-concave function has 
nonempty superdifferentials; in other words, it can be touched everywhere from above by a $C^1$ 
function; see \cite[Proposition 3.3.4]{CaS}. The condition \eqref{regular} is slightly weaker than this 
property in $\R^N$. 

\begin{prop}[Regularity and upper testability in Euclidean spaces]\label{prop upper test}
Let $x_0\in \R^N$ and assume that $u: \R^N\to \R$ is Lipschitz near $x_0$. If there exists a function 
$\psi$ differentiable at $x_0$ such that $u-\psi$ attains a local maximum at $x_0$, then 
\beq\label{pointwise regularity}
|\nabla u|(x_0)=|\nabla^- u|(x_0).
\eeq
\end{prop}

\begin{proof}
By assumptions, we have 
\[
u(x)-\psi(x)\leq u(x_0)-\psi(x_0)
\]
for any $x\in \Omega$ near $x_0$. It follows that 
\[
\begin{aligned}
&\min\{u(x)-u(x_0), 0\}\leq \min\{\psi(x)-\psi(x_0), 0\},\\
&\max\{u(x)-u(x_0), 0\}\leq \max\{\psi(x)-\psi(x_0), 0\}.
\end{aligned}
\]
These relations imply that 
\beq\label{eq regular1}
\begin{aligned}
&|\nabla^- u|(x_0)\geq |\nabla^- \psi|(x_0),\\
&|\nabla^+ u|(x_0)\leq |\nabla^+ \psi|(x_0).
\end{aligned}
\eeq
Noticing that $\psi$ is differentiable at $x_0$, we have 
\beq\label{differentiable}
|\nabla^+ \psi|(x_0)=|\nabla^- \psi|(x_0)=|\nabla \psi(x_0)|.
\eeq
We can use \eqref{eq regular1} to obtain \eqref{pointwise regularity} immediately.
\end{proof}

The statement converse to that of Proposition \ref{prop upper test} also holds when $N=1$ but it fails to hold in 
higher dimensions, as indicated by the example below. 
\begin{example}
Let $u: \R^2\to \R$ be given by
\[
u(x)=\min\{x_1, 0\}+|x_2|\quad \text{for $x=(x_1, x_2)\in \R^2$. }
\]
It is clear that $u$ is Lipschitz in $\R^2$ and $|\nabla u|(0)=|\nabla^- u|(0)=1$ but $u$ cannot be tested 
at $x=0$  from above by any function that is differentiable at $x=0$. 
\end{example}

However, for $x_0\in \R^N$, if \eqref{regular} holds in a neighborhood $\Omega$ of $x_0$ and $|\nabla u|$ 
is continuous in $\Omega$, then $u$ can be touched from above everywhere in $\Omega$ by a $C^1$ function. 
Indeed, under these assumptions, $u$ can be regarded as a Monge solution of \eqref{eikonal eq} with $f$ 
continuous in $\Omega$. This in turn implies that $u$ is an s-solution and a conventional viscosity solution 
in the Euclidean space by Corollary \ref{cor equiv Monge}.  Then its local semi-concavity and testability from 
above follow immediately.  

These suggest that \eqref{regular} can be adopted as a generalized semi-concavity in more general metric spaces.

\bibliographystyle{abbrv}

\end{document}